\documentclass[11pt,notitlepage,a4paper]{article}

\usepackage{amssymb} 
\usepackage[intlimits]{amsmath}
\usepackage{amsfonts}
\usepackage{amsthm} 

\usepackage{mathptmx}

\usepackage{hyperref}
\hypersetup{colorlinks=true,linkcolor=RoyalPurple,citecolor=RoyalPurple}

\usepackage{cite}

\usepackage{geometry}

\usepackage{color}
\usepackage[dvipsnames]{xcolor}

\let\oldsqrt\sqrt
\def\sqrt{\mathpalette\DHLhksqrt}
\def\DHLhksqrt#1#2{%
\setbox0=\hbox{$#1\oldsqrt{#2\,}$}\dimen0=\ht0
\advance\dimen0-0.2\ht0
\setbox2=\hbox{\vrule height\ht0 depth -\dimen0}%
{\box0\lower0.4pt\box2}}

\allowdisplaybreaks

\newcommand{\R}{\mathbb{R}} 
\newcommand{\N}{\mathbb{N}} 

\newcommand{\dist}{\textnormal{dist}} 
\newcommand{\diam}{\textnormal{diam}} 
\newcommand{\supp}{\textnormal{supp}} 
\newcommand{\diw}{\textnormal{div}} 
\newcommand{\pair}[2]{\langle\, #1 \,,\, #2 \,\rangle} 

\renewcommand{\phi}{\varphi}

\newcommand{\cD}{{\mathcal D}}
\newcommand{\cE}{{\mathcal E}}
\newcommand{\cF}{{\mathcal F}}
\newcommand{\cG}{{\mathcal G}}
\newcommand{\cH}{{\mathcal H}}

\newcommand{\cL}{{\mathcal L}}
\newcommand{\cM}{{\mathcal M}}

\newcommand{\cS}{{\mathcal S}}

\theoremstyle{definition}
\newtheorem{defi}{Definition}[section]
\newtheorem{remark}[defi]{Remark}

\theoremstyle{plain} 
\newtheorem{thm}[defi]{Theorem}
\newtheorem{prop}[defi]{Proposition}
\newtheorem{lemma}[defi]{Lemma}
\newtheorem{cor}[defi]{Corollary}

\theoremstyle{definition}

\numberwithin{equation}{section} 

\title{On the maximum principle for higher-order fractional Laplacians}
\author{
Nicola Abatangelo\footnote{Universit\'{e} Libre de Bruxelles CP 214, boulevard du Triomphe, 1050 Ixelles (Belgium) nicola.abatangelo@ulb.ac.be.}\hspace{0.5ex},
\ Sven Jarohs\footnote{Institut f\"ur Mathematik, Goethe-Universit\"at, Frankfurt, Robert-Mayer-Stra\ss e 10, D-60054 Frankfurt, jarohs@math.uni-frankfurt.de.}\hspace{0.5ex}, and
Alberto Salda\~{n}a\footnote{Universit\'{e} Libre de Bruxelles CP 214, boulevard du Triomphe, 1050 Ixelles (Belgium), asaldana@ulb.ac.be.}
}
\date{}
\begin{document}
\maketitle
\begin{abstract}
{ We study existence, regularity, and positivity of solutions to linear problems involving higher-order fractional Laplacians $(-\Delta)^s$ for any $s>1$. Using the nonlocal properties of these operators, we provide an explicit counterexample to general maximum principles for $s\in(n,n+1)$ with $n\in\mathbb N$ odd. In contrast, we show the validity of Boggio's representation formula for all integer and fractional powers of the Laplacian $s>0$. As a consequence, maximum principles hold for weak solutions in a ball.  Our proofs rely on a new variational framework based on bilinear forms, on characterizations of $s$-harmonic functions using higher-order Martin kernels, and on a differential recurrence equation for Boggio's formula.  We also discuss the case of the whole space, where maximum principles are a consequence of the fundamental solution.
}
\end{abstract}
{\footnotesize
\begin{center}
\end{center}
\begin{center}
\end{center}
}

\section{Introduction}\label{sec:introduction}

In the study of elliptic partial differential equations, most of the analysis has been focused on second order problems, which effectively describe  many natural phenomena. The available results on existence and qualitative properties in this setting have achieved a remarkable degree of sophistication, to a large extent due to very powerful analytic techniques derived from maximum principles, for instance, Harnack inequalities, Hopf Lemmas, and sub- and supersolutions methods.

The theory for elliptic higher-order (\emph{i.e.}, higher than 2) operators, on the other hand, is comparatively underdeveloped.  Some of the main difficulties that appear in their study is precisely the lack of maximum principles, the fact that the set of solutions is usually larger and more complex, and a much more subtle relationship between regularity of solutions, boundary conditions, and smoothness of the domain.

Nevertheless, higher-order operators appear in many important models coming, for instance, from continuum mechanics, biophysics, and differential geometry. They appear, for example, in the study of thin elastic plates, stationary surface diffusion flow, Paneitz-Branson equations, Willmore surfaces, suspension bridges, phase-transition, and membrane biophysics, see \cite{GGS10,PT2001} and references therein. The study of higher-order operators is also motivated by the understanding of basic questions in the theory of partial differential equations, to identify the key elements which yield existence, uniqueness, qualitative properties, and  regularity of solutions. 

The paradigmatic higher-order operator is given by powers of the Laplacian $(-\Delta)^m$, $m\in\mathbb N$, also known as the \emph{polyharmonic operator}. The validity and characterization of positivity preserving properties in this case is an active field of research and many basic questions are still open. 
For example, consider $m=2$, {\it i.e.}, the bilaplacian operator $\Delta^2 u=\Delta(\Delta\, u)$, for which maximum principles are known to be a very delicate issue and do \emph{not} hold in general. 
To obtain well-posedness in boundary value problems, 
the bilaplacian requires extra boundary conditions ({\it b.c.}). 
Two of the most common are \emph{Navier b.c.} $u=\Delta u=0$ on $\partial \Omega$ 
and \emph{Dirichlet b.c.} $u=\partial_\nu u=0$ on $\partial \Omega$.  
The case of the bilaplacian with Dirichlet b.c. is particularly delicate, 
and the geometry of the domain plays an essential role. 
It is known that $\Delta^2 u\geq 0$ in $\Omega$ and $u=\partial_\nu u=0$ on $\partial \Omega$ 
implies that $u\geq 0$ if $\Omega$ is a ball, for example, 
since the corresponding Green function can be computed explicitly in this case and it is nonnegative.
However, if $\Omega\subset\mathbb R^2$ is an ellipse with semi-axis 1 and $\frac{1}{5}$, 
then one can give an elementary counterexample (a polynomial of degree 7) 
showing that the maximum principle does not hold, see \cite{ST94}.  
Many other counterexamples are known in the literature, we refer to \cite{GGS10} 
and the references therein for a survey on positivity preserving properties 
for boundary value problems involving polyharmonic operators. 

\medskip
In this paper, we study the validity of positivity preserving properties for \emph{fractional powers} of the Laplacian $(-\Delta)^s$, $s>1$.   Some known results for this operator are the following\footnote{For publication, this paper was splitted into two parts \cite{AJS16b} and \cite{AJS16a}; specifically, the proofs of Theorem \ref{green:thm} and Proposition \ref{lem:sharm-intro} can be found in \cite{AJS16b}, whereas the proof of Theorem \ref{main:thm} is in \cite{AJS16a}. We also refer to \cite{AJS17a,AJS17b,ADFJS18} for more recent developments regarding higher-order powers of the Laplacian; in particular, reference \cite{AJS17a} focuses on explicit formulas for solutions of boundary value problems on balls, reference \cite{AJS17b} is a study of the different pointwise evaluations of $(-\Delta)^s$, and \cite{ADFJS18} is devoted to Dirichlet boundary value problems in the half-space.}.
General regularity results have been proved in \cite{G15:2}, 
a Poho\v zaev identity and an integration by parts formula is given in \cite{RS15}, 
a comparison between different higher-order fractional operators is done in \cite{MN15}, 
spectral results are obtained in \cite{G15}, 
and other aspects of nonlinear problems are considered in \cite{LM08, FW16b, PP15, MMS15}. 
Furthermore, the operator $(-\Delta)^s$ with $s\geq 1$ appears naturally in Geometry, 
for example, in the prescribed $Q-$curvature equation $(-\Delta)^{N/2}u=Ke^{Nu}$ \cite{CSY97,B14}. 

To begin our discussion on maximum principles, let us consider first the case $(-\Delta)^\sigma$ with $\sigma\in(0,1)$ and $u\in C^{\infty}_{c}(\R^{N})$, $N\in\N$. This operator is known as the \emph{fractional Laplacian} and it can be represented via the principal value integral
\begin{align}
(-\Delta)^{\sigma}u(x) &:=c_{N,\sigma}P.V.\int_{\R^{N}}\frac{u(x)-u(y)}{|x-y|^{N+2\sigma}}\ dy:=
c_{N,\sigma}\lim_{\epsilon\to0^+}\int_{|x-y|>\epsilon} \frac{u(x)-u(y)}{|x-y|^{N+2\sigma}}\ dy
\label{fraclaplace}
\end{align}
for $x\in \R^N$, where $c_{N,\sigma}=:=4^{\sigma}\pi^{-N/2}\sigma(1-\sigma)\frac{\Gamma(\frac{N}{2}+\sigma)}{\Gamma(2-\sigma)}$ 
is a normalization constant and $\Gamma$ denotes the Gamma function.  This operator is used to model \emph{nonlocal} interactions \cite{BV15,NPV11,T78}. Since $(-\Delta)^s$ is a nonlocal operator, boundary value problems are solved by prescribing boundary conditions in the whole complement of the domain (see \emph{\emph{e.g.}} \cite{K11}). In this case, as mentioned in \cite[Remark 4.2]{CS14}, the maximum principle holds in a weak setting for $\sigma\in(0,1)$ using the Dirichlet-to-Neumann extension from \cite{CS07} and testing the equation with $u^-:=-\min\{u,0\}$. 
This also follows directly from the nonlocal bilinear form
\begin{align*}
 \cE_{\sigma}(\varphi,\psi):=
\frac{c_{N,\sigma}}{2}\int_{\R^N}\int_{\R^N}\frac{(\varphi(x)-\varphi(y))(\psi(x)-\psi(y))}{|x-y|^{N+2\sigma}}\ dx\ dy=\int_{\R^N} |\xi|^{2\sigma}\cF \varphi (\xi)\cF \psi (\xi)\ d\xi,
\end{align*}
where $\cF$ denotes the Fourier transform, see \cite{JW14,sven:thesis}.  In particular, if $\Omega\subset\mathbb R^N$ is an open set, 
$u$ is in the fractional Sobolev space $H^s(\R^N)$, $u\geq0$ in $\R^N\setminus \Omega$, and $ \cE_{\sigma}(u,\varphi)\geq 0$ for all \emph{nonnegative} $\varphi\in H^\sigma(\R^N)$ with $\varphi\equiv 0$ in $\R^N\backslash\Omega$, then $u\geq0$ in $\Omega$. 

To study the higher-order case $s>1$ we extend this variational setting. Namely, fix $s=m+\sigma$ with $m\in \N$ and $\sigma\in(0,1)$. For $\Omega\subset \R^N$ open we define the fractional Sobolev space with zero boundary conditions
\begin{align}\label{Hs0:def}
\cH^{s}_0(\Omega)&:=\{u\in H^{s}(\R^N)\;:\; u\equiv 0\;\text{on $\R^N\setminus \Omega$}\}
\end{align}
equipped with the norm 
$\|u\|_{\cH^s_0(\Omega)}:=(\sum_{|\alpha|\leq m}\|\partial^{\alpha} u\|_{L^2(\Omega)}^2+\cE_{s}(u,u))^{\frac{1}{2}}$, 
where
\begin{equation}\label{bilin:def}
\cE_{s}(u,v):=\left\{\begin{aligned}
&\cE_{\sigma}(\Delta^{\frac{m}{2}}  u,\Delta^{\frac{m}{2}} v),&& \quad \text{if $m$ is even,}\\
&\sum_{k=1}^{N}\cE_{\sigma}(\partial_k \Delta^{\frac{m-1}{2}} u,\partial_k \Delta^{\frac{m-1}{2}} v),&&
\quad  \text{if $m$ is odd,}
\end{aligned}\right.
\end{equation}
for $u,v\in \cH^{s}_0(\Omega)$.
We now introduce the notion of weak solution. For $f\in L^2_{loc}(\Omega)$ we say that a function $u\in H^{s}(\R^N)$ is a \emph{weak supersolution} of 
\begin{equation}\label{wsol:def1}
    (-\Delta)^{s}u= f\quad\text{ in $\Omega$,}\qquad u=0 \quad\text{ on $\R^N\setminus \Omega$},
\end{equation}
if $u\geq 0$ on $\R^N\setminus \Omega$ and for all $\varphi\in \cH^{s}_0(\Omega)$ with compact support in $\R^N$ we have
\begin{align}\label{wsol:def}
\cE_{s}(u,\varphi)= \int_{\Omega}f(x)\varphi(x)\ dx.
\end{align}
We call $u\in H^{s}(\R^N)$ a \textit{weak subsolution} of \eqref{wsol:def1} if $-u$ is a weak supersolution of \eqref{wsol:def1}. If $u\in H^{s}(\R^N)$ is a weak super- and subsolution of (\ref{wsol:def1}), then we call $u$ a \emph{weak solution} of (\ref{wsol:def1}).

Our first result shows that the (weak) maximum principle does \emph{not} hold in general for weak solutions.
\begin{thm}\label{main:thm}
Let $N\in\N$, $D\subset\R^N$ be an open set, $s\in(k,k+1)$ for some $k\in\mathbb N$ odd, and let $A$ be a nonempty ball compactly contained in $\R^N\setminus D$. There is a smooth positive function $f\in C^\infty(\overline{\Omega})$ such that the problem \eqref{wsol:def1} in $\Omega=D\cup A$ admits a sign-changing weak solution $u\in\cH_0^{s}(\Omega)\cap C(\R^N)\cap C^{\infty}(\Omega)$ with $u\lneq 0$ in $D$ and $u\gneq 0$ in $A$.
\end{thm}

The proof of Theorem \ref{main:thm} is made via an explicit counterexample, which exploits the nonlocal nature of the operator and the fact that the domain is disconnected.  Although our approach to prove Theorem \ref{main:thm} cannot be used for $s\in(k,k+1)$ with $k\in\mathbb N$ even, we do not expect that general maximum principles hold for any $s>1$.  We refer to \cite{KKM89} for counterexamples involving even powers of the Laplacian and to \cite{S2016} for a counterexample to the trilaplacian, which seems to be the only available counterexample for odd powers.

Theorem \ref{main:thm} is particularly interesting for $s\in(1,\frac{3}{2})$, since in this case \cite[Th\'{e}or\`{e}me 1]{M89} implies that $u^-\in H^s(\Omega)$ if $u\in H^s(\Omega)$ and this is the main ingredient in the proof of maximum principles for $s\in(0,1]$, which uses $u^-$ as a test function. Indeed, the proof of Theorem \ref{main:thm} reveals that an essential role is played by the following simple fact due to integration by parts: for $u\in H^s(\R^N)$, $\phi\in C^\infty_c(\R^N)$, and $u,\phi\geq 0$ with $\supp\, u\cap \supp\, \phi=\emptyset$, we have that $\cE_s(u,\phi)<0$ if $s\in(0,1)$ and $\cE_s(u,\phi)>0$ if $s\in(k,k+1)$ with $k\in\mathbb N$ odd.  This is the main reason why the proof of maximum principles for $s\in(0,1)$ cannot be extended to $s\in(1,\frac{3}{2})$, see Remark \ref{negative_interaction}. Another consequence of this fact is the following remarkable property.
\begin{cor}\label{smoothgoutside}
 Let $m\in\N_0$, $\sigma\in(0,1)$, $s=m+\sigma$, $\Omega\subset \R^N$ be a smooth bounded domain, and $g\in C^{\infty}_c(\Omega)\backslash\{0\}$ be a nonnegative function, then $(-1)^{m+1}(-\Delta)^sg>0$ in $\R^N\setminus \overline{\Omega}$.
\end{cor}

Note that this is a purely nonlocal phenomenon. Moreover, a direct consequence of Theorem \ref{main:thm} is that maximum principles cannot hold for weak supersolutions in more general domains.
\begin{cor}\label{main:cor}
Let $\Omega\subset\R^N$ be an open set such that $\R^N\setminus \Omega$ has nonempty interior and let $s\in(k,k+1)$ for some $k\in\mathbb N$ odd. There is a weak supersolution $u\in H^s(\R^N)\backslash\{0\}$ of \eqref{wsol:def1} with $f\geq 0$ such that $u\lneq 0$ in $\Omega$.
\end{cor}

In particular, maximum principles for $(-\Delta)^s$ may only hold for \emph{solutions} and only in some domains. 


Next, we show that maximum principles for weak solutions hold on balls and are a consequence of an explicit representation formula. In the following, $\delta_y$ denotes the Dirac measure centered at $y\in\R^N$ and $C^r(B)=C^{n,l}(B)$ for $r=n+l$ with $n\in\N_0$ and $l\in(0,1]$.

\begin{thm}\label{green:thm}
Let $\sigma\in(0,1],$ $m\in\N$, $s=m+\sigma$, $N\in \N$, $B\subset \R^N$ the unitary ball, and let 
\begin{equation}\label{greenball-intro}
 \cG_s(x,y) :=k_{N,s} |x-y|^{2s-N}\int_0^{\rho(x,y)}\frac{v^{s-1}}{(v+1)^\frac{N}{2}}\ dv \qquad\text{ for }x,y\in\R^N,\ {x\neq y},
\end{equation}
where
\begin{equation}\label{greenconst-intro}
\rho(x,y):=\frac{(1-|x|^2)_+(1-|y|^2)_+}{|x-y|^2},\qquad k_{N,s}:=\frac{\Gamma(\frac{N}{2})}{\pi^\frac{N}{2}4^s\Gamma(s)^2}.
\end{equation}
Then $\cG_s(\cdot,y)$ is a distributional solution of $(-\Delta)^s v=\delta_y$ in $B$ for every $y\in B$. Moreover, if $f\in C^{\alpha}(B)$  for some $\alpha\in(0,1)$ with $2s+\alpha\not\in \N$ and
\begin{align}\label{Gsu}
u:\R^N\to\R\quad \text{ is given by}\quad u(x)&:=\ \int_{B} \cG_{s}(x,y)\,f(y)\ dy,
\end{align}
then $u\in C^{2s+\alpha}_{loc}(B)\cap C_0^s(B)\cap \cH^s_0(B)$ is the unique weak solution of \eqref{wsol:def1} with $\Omega=B$. Furthermore, $(-\Delta)^{m}(-\Delta)^{\sigma}u(x)=f(x)$ pointwise for every $x\in B$, where the fractional Laplacian $(-\Delta)^\sigma u$ is evaluated as in \eqref{fraclaplace}, and there is $C>0$ such that 
\begin{align}\label{decay}
\|\operatorname{dist}(\cdot,\partial B)^{-s}u\|_{L^{\infty}(B)}<C\|f\|_{L^{\infty}(B)}\qquad \text{ for $s\geq 1$}. 
\end{align}
\end{thm}

The function $\cG_s$ is known as \emph{Boggio's formula}, see \cite{B16,GGS10,DG2016}. The proof of Theorem \ref{green:thm} is based on a differential recurrence formula for $\cG_s$ in terms of $\cG_{s-1}$ and an explicit function $P_{s-1}$ which is $(s-1)$-harmonic in the ball, see Lemma \ref{Lap:green} below.  Since the validity of Boggio's formula is known for $s\in(0,1]$, this allows us to implement an induction argument to extend this result to all $s>1$. We remark that our approach also provides an alternative proof for $s\in \N$. Two key elements in the proof are an elementary {\textemdash but lengthy\textemdash} pointwise calculation of $-\Delta_x\cG_s(x,y)$ for $y\neq x$ and $s>1$ (see Lemma \ref{Lap:green}) and the introduction of {\it higher-order Martin kernels}
\[
M_s(x,\theta)\ =\ \lim_{y\to\theta}\frac{\cG_s(x,y)}{{(1-|y|^2)}^s}\qquad \text{ for }x\in\R^N,\ \theta\in\partial B,
\]
which we use to characterize a large class of $s$-harmonic functions, see Proposition \ref{lem:sharm-intro} below.  Martin kernels were introduced in \cite{M41} for $s=1$ to provide an analogue of Poisson kernels in nonsmooth domains and in \cite{B99} for $s\in(0,1)$ to give representation formulas for $s$-harmonic functions which are singular at the boundary of the domain (a purely nonlocal phenomenon). Our construction is similar to the one presented in \cite{nicola} and we generalize it to $s>1$. See also Lemma \ref{Martin:ball} for a simplified expression of $M_s$.

With these elements we show first that $u$ given as in \eqref{Gsu} is a distributional solution and the order of derivation $(-\Delta)^{m}(-\Delta)^{\sigma}u$ appears as a consequence of integration by parts, see Lemma \ref{ibyp}. This order, however, may be partially interchanged depending on the interior and boundary regularity of $u$, see Proposition \ref{interchange-der}. For example, if $f\in C^{\alpha}(B)$, $m$ is even, and $u$ is as in \eqref{Gsu}, then $(-\Delta)^{m}(-\Delta)^{\sigma}u=(-\Delta)^{\frac{m}{2}}(-\Delta)^{\sigma} (-\Delta)^{\frac{m}{2}}u$ pointwise in $B$, which is consistent with the variational framework described above.

Note that the regularity of solutions \textemdash in particular, integrability, which is used to show uniqueness \textemdash is more involved for higher-order fractional powers of the Laplacian. For instance, consider the function $u(x)=(1-|x|^2)_+^s$ for $s>0$, which is a pointwise solution of $(-\Delta)^s u = C$  in $B$ for some constant $C>0$ (see Corollary \ref{sol:ball} below). Clearly $u$ belongs to $H^{2s}(B)$ if $s$ is an integer, since in this case $u$ is a polynomial. For general $s$, however, $u$ may have derivatives which blow-up at the boundary, for example terms involving $(1-|x|^2)_+^{s-2}$ are \emph{not} in $L^2(B)$ if $s\in(1,\frac{3}{2})$. To circumvent this difficulty and show that $u\in \cH_0^s(B)$, we use standard interpolation theory as in \cite{T78,Lototsky}. 

In the recent work \cite{DG2016} the authors show independently the validity of Boggio's formula for all $s>0$ considering only smooth functions with compact support as right-hand sides.  The proofs in \cite{DG2016} are very different from ours and rely on covariance under M\"{o}bius transformations and computations using Hypergeometric functions, see also \cite[Remark 1]{DKK15}. 

Our approach also provides the following new insights on higher-order $s$-harmonic functions and on distributional solutions satisfying different boundary conditions.
\begin{prop}\label{lem:sharm-intro}
Let $s>0$ and $\mu$ be a finite Radon measure on $\partial B$. The function
\[
u(x)=\int_{\partial B}M_s(x,z)\;d\mu(z)\qquad \text{ for } x\in\R^N
\]
is $s$-harmonic in $B$ in the sense of distributions.
\end{prop}

Proposition \ref{lem:sharm-intro} was known only for $s\in(0,1)$, see \cite{nicola,B99}.  See also Remark \ref{rmk:sharm} for more on $s$-harmonic functions. The proof of Proposition \ref{lem:sharm-intro} follows directly from Theorem \ref{green:thm} and Lemma \ref{lem:sharm}. 

\begin{cor}\label{green:cor2}
Let $s>1$, $j\in (0,s)\cap\N$, and $\mu$ be a finite Radon measure on $B$. Then the function $u_j:\R^N\to\R,$ given by $u_j(x)=\int_B \cG_{s-j}(x,y)\int_B \cG_j(y,z)\ d\mu(z) dy$ is a distributional solution of $(-\Delta)^s u_j=\mu$. In particular, if $d\mu(z)=f(z)\ dz$ for some $f\in C^{\alpha}(B)$ then $u_j\in C^{s-j}_0(B)$ is a distributional solution of $(-\Delta)^s u_j=f$.
\end{cor}
Note that the solutions given by Corollary \ref{green:cor2} are not the one given by Theorem~\ref{green:thm}, in particular they correspond to different boundary conditions and do not satisfy \eqref{decay}. With these solutions we can construct the following $s$-harmonic functions.
\begin{cor}\label{green:cor}
For $s>1$, $x,y\in B$, $x\neq y$, let $v(x,y):=\cG_s(x,y)-\int_B\cG_1(x,z)\cG_{s-1}(z,y)\ dz$. Then, for fixed $y\in B$ (resp. $x\in B$), $v$ is $s$-harmonic with respect to $x$ (resp. $y$) in $B$ in the sense of distributions.
\end{cor}

Finally, our method also provides information on the sign of some $s$-harmonic functions.
\begin{cor}\label{harm:cor}
Fix $s\in(k,k+1)$ for some $k\in\N$ odd, $B\subset \R^N$ the unitary ball, and $g\in C^\infty_c(\R^N\setminus\overline{B})$ with $g\geq 0$. Then, there exists a unique weak solution $u\in H^s(\R^N)$ to 
$(-\Delta)^s u=0$ in $B$ with $u=g$ in $\R^N\setminus B$. Moreover, $u\leq 0$ in $B$.
\end{cor}


As a second example where maximum principles are satisfied, we discuss in Theorem \ref{ppprn} below the case of the whole space. Moreover, we show the existence of distributional solutions to $(-\Delta)^su=f$ in $\R^N$ for \emph{all} $s>0$ in Corollary \ref{generalcase}. Note that the fundamental solution is \emph{not} given by the Riesz kernel if $s-\frac{N}{2}\in \N_0$, see Definition \ref{FNs:def}.

\medskip

The organization of the paper is the following. The notation used throughout the paper is introduced in Section \ref{Notation} and the development of the variational framework for higher-order fractional operators can be found in Section \ref{s12}. The proofs of Theorem \ref{main:thm} and Corollaries \ref{smoothgoutside} and \ref{main:cor} are contained in Section \ref{ce:sec}. In Section \ref{sec:greenfund} we discuss the distributional setup of the problem and provide a representation formula for solutions in the whole space for all $s>0$. The proofs of Theorem \ref{green:thm} and Corollaries \ref{green:cor2}, \ref{green:cor}, and \ref{harm:cor} are written in Section \ref{ballcase} together with some remarks on $s$-harmonic functions. Finally, in the Appendix, we prove a differential recurrence equation involving Boggio's formula and we present results regarding the interchange of derivatives.

\subsection*{Acknowledgements}

We are very thankful to Denis Bonheure, Roberta Musina, Hans Triebel, and Tobias Weth for valuable discussions.

\section{Notation}\label{Notation}

Let $N\in \N$ and $U,D\subset \R^N$ be nonempty measurable sets. We denote by $1_U: \R^N \to \R$ the characteristic function, $|U|$ the Lebesgue measure, and $\diam(U)$ the diameter of $U$. The notation $D \subset \subset U$ means that $\overline D$ is compact and contained in the interior of $U$. The distance between $D$ and $U$ is given by $\dist(D,U):= \inf\{|x-y|\::\: x \in D,\, y \in U\}$ and if $D= \{x\}$ we simply write $\dist(x,U)$. Note that this notation does {\em not} stand for the usual Hausdorff distance. For $x \in \R^N$ and $r>0$ let $B_r(x)$ denote the open ball centered at $x$ with radius $r$, moreover we fix $B:=B_1(0)$, $\omega_N=|B|$, and $d(x)=\dist(x,\R^N\setminus B)$ for $x\in \R^N$.

If $u$ is in a suitable function space, we use $\cF u$ or $\widehat u$ to denote the Fourier transform of $u$ and $\cF^{-1}(u)$ or $u^\vee$ to denote its inverse. 

For any $s\in\R$, we define $H^s(\R^N):=\left\{u\in L^2(\R^N)\;:\; (1+|\xi|^{2})^{\frac{s}{2}}\ \widehat u\in L^2(\R^N)\right\};$ moreover,
if $U$ is open, we define $\cH^{s}_0(U)$ as in \eqref{Hs0:def} and, if $U$ is smooth, we put $H^s(U):=\{u1_{U}\;:\; u\in H^s(\R^N)\}$.

We use $\cS$ to denote the space of Schwartz functions in $\R^N$ and $\cS'$ its dual (the space of tempered distributions) and denote $\pair{\cdot}{\cdot}:\cS'\times \cS\to \R$ the dual pairing of $\cS'$ and $\cS$. For the definition of these spaces and basic properties we refer to \cite[Chapter 2.3]{G08}. Recall that $\pair{\widehat u}{f}=\pair{u}{\widehat f}$ for all $f\in\cS$.  As usual, for suitable $u:\R^N\to\R$ we identify $u$ with its associated distribution $T_u:\cS\to\R^N$ given by $\pair{T_u}{f}=\int_{\R^N} u(x)f(x)\ dx$ for all $f\in \cS$.

For $m\in \N_0$, $\sigma\in[0,1)$, $s=m+\sigma$, and $U$ open, we write $C^s(U):=C^{m,\sigma}(U)$ (resp. $C^s(\overline{U})$) to denote the space of $m$-times continuously differentiable functions in $U$ (resp. $\overline{U}$) and, if $\sigma>0$, whose derivatives of order $m$ are $\sigma$-H\"older continuous in $U$. Moreover, for $s\in[0,\infty]$, $C^s_c(U):=\{u\in C^s(\R^N): \supp \ u\subset\subset U\}$ and $C^s_0(U):=\{u\in C^s(\R^N): u\equiv 0 \text{ on $\R^N\setminus U$}\}$, where $\supp\ u:=\overline{\{ x\in U\;:\; u(x)\neq 0\}}$ is the support of $u$.

Recall \eqref{bilin:def}. If $m\in \N$ is odd we also use the following vector notation
\[
\cE_{\sigma}(\nabla (-\Delta)^{\frac{m-1}{2}}u ,\nabla (-\Delta)^{\frac{m-1}{2}}u ):= \sum_{k=1}^{N}\cE_{\sigma}(\partial_k(-\Delta)^{\frac{m-1}{2}}u ,\partial_k (-\Delta)^{\frac{m-1}{2}}u )=\cE_{s}(u,u).
\]

Let $u: U \to \R$ be a function. We use $u^+:=u_+:= \max\{u,0\}$ and $u^-:=-\min\{u,0\}$ to denote the positive and negative part of $u$ respectively.

Finally, $\Gamma$ denotes the standard \emph{Gamma function} and if $f: U\times D\to \R$ we write $(-\Delta_x)^s f(x,y)$ to denote derivatives with respect to $x$, whenever they exist in some appropriate sense.

\section{Variational framework} \label{s12}

Let $\Omega\subset \R^N$ be an open set, and fix $m\in \N_0:=\{0,1,2,\ldots\}$, $\sigma\in(0,1)$, and $s=m+\sigma$. Recall the space $\cH^s_0(\Omega)$ as defined in \eqref{Hs0:def} equipped with the bilinear form $\cE_{s}(\cdot,\cdot)$ defined in \eqref{bilin:def}.  We begin by showing the equivalence between the definition of weak solution (see \eqref{wsol:def}) and the definition of solution via the Fourier transform $\cF$.

\begin{prop}\label{justification}
Let $f\in L^2(\Omega)$. The function $u\in H^{s}(\R^N)$ is a weak supersolution of \eqref{wsol:def1} if and only if
\begin{align*}
\int_{\R^N} |\xi|^{2s}\cF u(\xi)\cF \varphi(\xi)\ d\xi\geq \int_{\R^N} f(x)\varphi(x)\ dx
\end{align*}
for all nonnegative $\varphi\in \cH_0^{s}(\Omega)$ with compact support in $\mathbb R^N$. Moreover, for $u\in H^{2s}(\R^N)$ the operator $(-\Delta)^{s}u:=\cF^{-1}(|\cdot|^{2s} \cF u)$ is well-defined in $L^2(\R^N)$ and we have
\[
\cE_{s}(u,\phi)=\int_{\R^N}(-\Delta)^{s}u(x)\varphi(x)\ dx\quad\text{ for all $\varphi\in H^{s}(\R^N)$.}
\]
\end{prop}

\begin{proof}
Let $\varphi\in \cH_0^{s}(\Omega)$ be nonnegative and $u\in H^{s}(\R^N)$. 
If $m$ is even, then
\begin{align*}
&\int_{\R^N} |\xi|^{2s}\cF u(\xi)\cF \varphi(\xi)\ d\xi= \int_{\R^N} |\xi|^{s} \cF u(\xi)\cdot |\xi|^{s}\cF\varphi(\xi)\ d\xi\\
&= \int_{\R^N}(-\Delta)^{\frac{\sigma}{2}} \Delta^{\frac{m}{2}} u(x)\cdot (-\Delta)^{\frac{\sigma}{2}} \Delta^{\frac{m}{2}} \varphi(x)\ dx\\
&=\frac{c_{N,\sigma}}{2}\int_{\R^N}\int_{\R^N}\frac{(\Delta^{\frac{m}{2}} u(x)-\Delta^{\frac{m}{2}} u(y))\cdot(\Delta^{\frac{m}{2}}\varphi(x)-\Delta^{\frac{m}{2}}\varphi(y))}{|x-y|^{N+2\sigma}}\ dxdy.
\end{align*}
And if $m$ is odd, then
\begin{align*}
&\int_{\R^N} |\xi|^{2s}\cF u(\xi)\cF\varphi(\xi)\ d\xi= \int_{\R^N} |\xi|^{s-1} (-i) \xi\cF u(\xi)\cdot \ i\xi|\xi|^{s-1}\cF\varphi(\xi)\ d\xi\\
&=\int_{\R^N} |\xi|^{s-1} (- i) \xi\cF u(\xi)\cdot \overline{(-i\xi|\xi|^{s-1}\cF\varphi(\xi))}\ d\xi\\
&= \int_{\R^N}(-\Delta)^{\sigma/2} \nabla \Delta^{\frac{m-1}{2}} u(x)\cdot (-\Delta)^{\sigma/2} \nabla\Delta^{\frac{m-1}{2}} \varphi(x)\ dx\\
&=\frac{c_{N,\sigma}}{2}\int_{\R^N}\int_{\R^N}\frac{(\nabla\Delta^{\frac{m-1}{2}}u(x)-\nabla\Delta^{\frac{m-1}{2}} u(y))\cdot(\nabla\Delta^{\frac{m-1}{2}} \varphi(x)-\nabla \Delta^{\frac{m-1}{2}}\varphi(y))}{|x-y|^{N+2\sigma}}\ dxdy.
\end{align*}

This proves the first part. If, in addition, $u\in H^{2s}(\R^N)$, then 
\begin{align*}
\int_{\R^N}|(-\Delta)^{s}u(x)|^2\ dx=\int_{\R^N}|\xi|^{4s}\left|\cF u(\xi)\right|^2\ d\xi=\cE_{2s}(u,u)<\infty,
\end{align*}
by standard properties of the Fourier transform. Now the last part follows from the above calculations.
\end{proof}

\begin{remark}\label{commutation}
If $u\in H^{2s}(\R^N)$ then it follows from the proof of Proposition \ref{justification} that 
\[
(-\Delta)^{s}u=(-\Delta)^{m}(-\Delta)^{\sigma}u=(-\Delta)^{\sigma}(-\Delta)^{m}u=\left\{\begin{aligned}
& (-\Delta)^{\frac{m}{2}}(-\Delta)^{\sigma}(-\Delta)^{\frac{m}{2}}u&\text{for $m$ even}\\
&\diw (-\Delta)^{\frac{m-1}{2}}(-\Delta)^{\sigma}(-\Delta)^{\frac{m-1}{2}} \nabla u &\text{for $m$ odd}
\end{aligned}\right.
\]
where $(-\Delta)^{\sigma}$ is defined as in (\ref{fraclaplace}) (see also Proposition \ref{interchange-der} for a general statement on the interchange of derivatives).
\end{remark}

\subsection{Poincar\'e Inequality and principal eigenvalues}

The following shows that $\cE_{s}$ satisfies a Poincar\'{e}-type inequality in bounded domains. This yields that $\cE_{s}$ is a scalar product and that $(\cH^{s}_0(\Omega),\cE_{s})$ is a Hilbert space. Let $\lambda_{1,s}=\lambda_{1,s}(\Omega)$ and $\lambda_{1,1}=\lambda_{1,1}(\Omega)$ denote the first eigenvalue of $((-\Delta)^s,\cH_0^s(\Omega))$ and of $(-\Delta,H^1_0(\Omega))$ respectively.

\begin{prop}[Poincar\'e inequality]\label{l2}
Let $\Omega\subset \R^N$ be an open and bounded set with Lipschitz boundary. For all $u\in \cH^{s}_0(\Omega)$ we have that
\[
\cE_{s}( u, u)\geq \lambda_{1,s}\|u\|_{L^2(\Omega)}^2 \qquad\text{ and }\qquad \cE_{s}( u, u)\geq \left\{\begin{aligned}
&\lambda_{1,\sigma}\|\Delta^{\frac{m}{2}} u\|_{L^2(\Omega)}^2 && \text{ if $m$ is even}\\
&\lambda_{1,\sigma}\|\nabla \Delta^{\frac{m-1}{2}} u\|_{L^2(\Omega)}^2 && \text{ if $m$ is odd,}
\end{aligned}\right.
\]
where
\begin{align}\label{inf}
 \lambda_{1,s}=\lambda_{1,s}(\Omega):=\min_{u\in \cH^{s}_0(\Omega)\backslash \{0\}}\frac{\cE_s(u,u)}{\|u\|_{L^2(\Omega)}^2}>0,
\end{align}
$\lambda_{1,s}\geq \lambda_{1,1}^\frac{m}{2}\lambda_{1,\sigma}$ if $m$ is even, and $\lambda_{1,s}\geq \lambda_{1,1}^\frac{m+1}{2}\lambda_{1,\sigma}$ if  $m$ is odd.  In particular, $\lim\limits_{r\to0}\inf\limits_{|\Omega|=r} \lambda_{1,s}(\Omega)=\infty$. Moreover, $(\cH^{s}_0(\Omega),\cE_{s}(\cdot,\cdot))$ is a Hilbert space.
\end{prop}
\begin{proof}
Let $u\in \cH^{s}_0(\Omega)$ and $m$ even. By standard estimates we have
\[
\cE_{\sigma}((-\Delta)^{\frac{m}{2}} u,(-\Delta)^{\frac{m}{2}} u)\geq \lambda_{1,\sigma}\|(-\Delta)^{\frac{m}{2}} u\|_{L^2(\Omega)}^2\geq  \lambda_{1,1}^{\frac{m}{2}}\lambda_{1,\sigma}\|u\|_{L^2(\Omega)}^2.
\]
Clearly this also implies that $\cE_{1+\sigma}$ is a scalar product and \eqref{inf} follows.  The case $m$ odd is analogous.

\medskip

We now prove that  $\cH^{s}_0(\Omega)$ is complete with respect to $\cE_{s}$. Let $(u_n)_n\subset \cH^{s}_0(\Omega)$ be a Cauchy sequence with respect to $\cE_{s}$. Hence by the above inequality it follows that $u_n\to u\in L^2(\Omega)$ for $n\to \infty$, where we use $L^2(\Omega)=\{u\in L^2(\R^N)\;:\; u= 0 \text{ on $\R^N\setminus \Omega$}\}$. Thus there is a subsequence $(u_{n_k})_k$ such that $u_{n_k}\to u$ a.e. in $\Omega$ as $k\to\infty$. By Fatou's Lemma we have
\[
\cE_{s}(u,u)\leq \liminf_{k\to\infty}\cE_{s}( u_{n_k}, u_{n_k})\leq \sup_{k\in \N}\cE_{s}( u_{n_k}, u_{n_k})<\infty,
\]
so that $u\in \cH^{s}_0(\Omega)$. Again by Fatou's Lemma we have for any $k\in \N$
\[
\cE_{s}(u-u_{n_k},u-u_{n_k})\leq \liminf_{j\to\infty}\cE_{s}(u_{n_j}- u_{n_k},u_{n_j}- u_{n_k})\leq \sup_{j\geq k}\cE_{s}(u_{n_j}- u_{n_k},u_{n_j}- u_{n_k})<\infty
\]
which gives $u_{n_k}\to u$ in $\cH^{s}_0(\Omega)$ for $k\to\infty$ since $(u_{n_k})_k$ is a Cauchy sequence with respect to $\cE_{s}$. This shows the completeness.
\end{proof}

\begin{remark}\label{generall2}
 The assumption on the Lipschitz regularity of the boundary in Proposition \ref{l2} can be removed if one argues instead with the Sobolev embedding of $H_0^{m}(\Omega)$ into $L^2(\Omega)$, but in this case the estimates for  $\lambda_{1,s}$ are not clear, since they rely on integration by parts.
\end{remark}

\begin{remark}For $\Omega$ smooth and $m=1$ we have the strict inequality $\lambda_{1,s}=\lambda_{1,1+\sigma}>\lambda_{1,1}\lambda_{1,\sigma}$. Indeed, let $A_su:=\sum_{i\in\mathbb N}a_i(u)\lambda_{i,1}^se_i$ denote the \emph{spectral} fractional Laplacian, where $e_i$ and $\lambda_{i,1}>0$ are the eigenfunctions and eigenvalues of $-\Delta$ in $H_0^1(\Omega)$ and $a_i(u):=\int_\Omega u e_i\ dx$ is the projection of $u$ in the direction $e_i$, see \cite{SV14,MN15}. We introduce also the following associated quadratic forms as in \cite{MN15},
\begin{align*}
 Q^D_s[u]&:= \int_{\mathbb R^N} |\xi|^{2s}|\cF u(\xi)|^2\ d\xi, &&u\in \operatorname{Dom}(Q^D_s):=\{u\in {\cal S}'(\R^N): Q^D_s[u]<\infty,\ \operatorname{supp}(u)\subset \overline{\Omega}\},\\
 Q^N_s[u]&:= \sum_{j\in\mathbb N} \lambda_{j,1}^s a_i(u)^2, &&u\in \operatorname{Dom}(Q^N_s):=\{u\in {\cal S}'(\R^N): Q^N_s[u]<\infty\},
\end{align*}
where ${\cal S}'$ denotes the space of distributions. Then, by \cite[Theorem 1 and Lemma 2]{MN15} we have that $Q^D_s[u]>Q^N_s[u]$ and $\operatorname{Dom}(Q^D_s)\subset \operatorname{Dom}(Q^N_s)$ for $s\in(1,2)$. Thus
\begin{align*}
 \lambda_{1,s}= \inf_{u\in \operatorname{Dom}(Q^D_s)} Q^D_s[u]\geq \inf_{u\in \operatorname{Dom}(Q^N_s)} Q^N_s[u]=\lambda_{1,1}^{s},
\end{align*}
since the first eigenvalue of $A_s$ is given by $\lambda_{1,1}^s$, as it is easily seen from the definition of $A_s$.  Furthermore, $\lambda_{1,\sigma}<(\lambda_{1,1})^\sigma$ for $\sigma\in(0,1)$ by \cite[Theorem 1]{SV14}. Thus, if $s=1+\sigma$ we have that $\lambda_{1,s}\geq(\lambda_{1,1})^{s}=\lambda_{1,1}(\lambda_{1,1})^\sigma>\lambda_{1,1}\lambda_{1,\sigma}$, as claimed.
\end{remark}

An immediate consequence of Proposition \ref{l2} and Remark \ref{generall2} is the following.
 
\begin{cor}\label{unique:weak}
Let $\Omega\subset \R^N$ be an open bounded set. Then for any $f\in L^2(\Omega)$ there is a unique weak solution $u\in \cH_0^{s}(\Omega)$ of $(-\Delta)^{s}u=f$ in $\Omega$.
\end{cor}
\begin{proof}
The statement follows from Riesz Theorem, since $\cE_{s}$ is a scalar product on $\cH_0^{s}(\Omega)$ by Proposition \ref{l2} and Remark \ref{generall2}.
\end{proof}

\subsection{Properties with respect to smooth functions}

\begin{lemma}\label{lem:smooth-to-weak}
Let $\Omega\subset \R^N$ open. Then $C^{s+\epsilon}_c(\Omega)\subset \cH^{s}_0(\Omega)$ for every $\epsilon>0$.
\end{lemma}
\begin{proof}
Let $m$ be even and without loss of generality assume that $\epsilon\in(0,1-\sigma]$. Let $f\in C^{m,\sigma+\epsilon}_c(\Omega)$ and $D:=\supp (f)$. There is $C>0$ such that 
\begin{align*}
|(-\Delta)^{\frac{m}{2}} f(x)-(-\Delta)^{\frac{m}{2}} f(y)|^2\leq C|x-y|^{2\sigma+2\epsilon}\quad \text{ and }\quad |f(x)|^2\leq C\quad \text{ for all $x,y\in \R^N$.}
\end{align*}
Let $R>0$ so that $D\subset\subset U:=B_R(0)$ and $\dist(D, \R^N\setminus U)\geq 1$. Then
\begin{align*}
 \cE_{\sigma}((-\Delta)^{\frac{m}{2}} f,(-\Delta)^{\frac{m}{2}} f)&\leq C \int_{U}\int_{U}|x-y|^{2\epsilon-N}\ dxdy +2C\int_D \int_{\R^N\setminus U}|x-y|^{-N-2\sigma}\ dxdy
<\infty.
\end{align*}
The case $m$ odd follows similarly.
\end{proof}

\begin{lemma}\label{existence}
Let $\Omega\subset\R^N$ be open and $u\in {C^{2m+2}_c(\Omega)}$. Then 
 \begin{align*}
  \cE_{s}(u,v)=\int_{\Omega}(-\Delta)^{s}u(x) v(x)\ dx\qquad \text{ for all $v\in \cH_0^{s}(\Omega)$.}
 \end{align*}
\end{lemma}
\begin{proof} 
This is a consequence of Proposition \ref{justification} and Lemma \ref{lem:smooth-to-weak}. A direct proof can also be done using integration by parts if $\Omega$ has Lipschitz boundary.
\end{proof}

We now introduce the space $S^{k}_{s}$, which allows us to estimate pointwise fractional Laplacians, cf. \cite[Section 2]{FW15}.  For $s>0$ and $k\in \N$ let 
\[
S^{k}_{s}:=\{\phi\in C^k(\R^N)\;:\; \sup_{x\in \R^N}(1+|x|^{N+2s})\sum_{|\alpha|\leq k}|\partial^{\alpha}\phi(x)| <\infty\}
\]
endowed with the norm $\|\phi\|_{k,s}:= \sup\limits_{x\in \R^N}(1+|x|^{N+2s})\sum\limits_{|\alpha|\leq k}|\partial^{\alpha}\phi(x)|$. In particular, $\cS\subset S^{k}_{s}$.

\begin{lemma}\label{boundedevaluationexistence}
Let $\sigma\in(0,1]$, $m\in \N_0$, and $s=m+\sigma$. There is ${C=C(N,m,\sigma)>0}$ such that
\begin{align}\label{bd:eq}
 |(-\Delta)^{s}f(x)|\leq C\frac{\|f\|_{2m+2,s}}{1+|x|^{N+2s}}\qquad \text{  for every $f\in S^{2m+2}_{s}$ and for all $x\in \R^N$}.
\end{align}
\end{lemma}
\begin{proof}
If $\sigma=1$, then \eqref{bd:eq} follows by definition with $C=1$. For the rest of the proof, we denote by $C>0$ possibly different constants depending only on $N,$ $m,$ and $\sigma$. Let $\sigma\in(0,1)$ and note that $(-\Delta)^{m+\sigma} f= (-\Delta)^{\sigma}(-\Delta)^{m} f$ by Remark \ref{commutation}. To simplify the notation let $\varphi:=(-\Delta)^m f$ and recall that $B:=B_1(0)$.  For $x\in\R^N$ we have, by the Mean Value Theorem (see Lemma \ref{diffrep}),
\begin{align}\label{A5}
|&(-\Delta)^{\sigma+m} f(x)|=\frac{c_{N,\sigma}}{2}\Bigg|\int_{\R^N} \frac{2\phi(x)-\phi(x+y)-\phi(x-y)}{|y|^{N+2\sigma}}\ dy\Bigg|\notag\\
&\leq C\int_{B}\int_0^1\int_0^1\frac{|H_{\phi}( x+(t-\tau)y)|}{|y|^{N+2\sigma-2}}\ d\tau dt dy+\Bigg|\int_{\R^N\setminus B} \frac{2\phi(x)-\phi(x+y)-\phi(x-y)}{|y|^{N+2\sigma}}\ dy\Bigg|=:f_1+f_2.
\end{align}
Note that 
\begin{align}\label{A4}
f_1&\leq C\|f\|_{2m+2,s}\int_{B}\int_0^1\int_0^1\frac{|y|^{-N-2\sigma+2}}{1 +|x+(t-\tau)y|^{N+2s}}\ d\tau dt dy  \leq C\frac{\|f\|_{2m+2,s}}{1+|x|^{N+2s}},\\
 f_2&\leq 2\int_{\R^N\setminus B} \frac{|\phi(x)|}{|y|^{N+2\sigma}}\ dy+2\Bigg|\int_{\R^N\setminus B} \frac{\phi(x+y)}{|y|^{N+2\sigma}}\ dy\Bigg|
 \leq C\frac{\|f\|_{2m+2,s}}{1+|x|^{N+2s}}+2\Bigg|\int_{\R^N\setminus B} \frac{\phi(x+y)}{|y|^{N+2\sigma}}\ dy\Bigg|.\label{A3}
\end{align}
Using integration by parts $m-$times we obtain
\begin{align}\label{A2}
 \Bigg |\int_{\R^N\setminus B} \frac{\phi(x+y)}{|y|^{N+2\sigma}}\ dy\Bigg|=\Bigg |\int_{\R^N\setminus B} \frac{(-\Delta)^m f(x+y)}{|y|^{N+2\sigma}}\ dy\Bigg|\leq C\frac{\|f\|_{2m+2,s}}{1+|x|^{N+2s}}+ C\int_{\R^N\setminus B} \frac{|f(x+y)|}{|y|^{N+2\sigma+2m}}\ dy.
\end{align}
Moreover,
\begin{align}\label{A1}
 \int_{\R^N\setminus B} \frac{|f(x+y)|}{|y|^{N+2\sigma+2m}}\ dy\leq \frac{\|f\|_{2m+2,s}}{1+|x|^{N+2s}}\int_{\R^N\setminus B} \frac{1+|x|^{N+2s}}{(1+|x+y|^{N+2s})|y|^{N+2s}}\ dy
\end{align}

By \eqref{A5}-\eqref{A1} it suffices to show that there is $C>0$ depending only on $N,$ $m,$ and $\sigma$ such that
\begin{align}\label{A6}
\int_{\R^N\setminus B} \frac{1+|x|^{N+2s}}{(1+|x+y|^{N+2s})|y|^{N+2s}}\ dy<C
\end{align}
for all $x\in\R^N$.  If $|x|<2$ then \eqref{A6} follows by taking the maximum over $x\in 2B$. We now argue as in \cite[Lemma 2.1]{FW15}.  Fix $|x|\geq 2$ and let $U:=\{y\in\R^N\backslash B : |x+y|\geq \frac{|x|}{2}\}$. If $y\in U$ then $1+|x|^{N+2s}\leq C(1+|x+y|^{N+2s})$ and if $y\in \R^N\backslash U$ then $|y|>\frac{|x|}{2}$.  Thus,
\begin{align*}
&\int_{U} \frac{1+|x|^{N+2s}}{(1+|x+y|^{N+2s})|y|^{N+2s}}\ dy \leq C \int_{\R^N\backslash B} |y|^{-N-2s}\ dy<C,\\
 &\int_{\R^N\backslash U} \frac{1+|x|^{N+2s}}{(1+|x+y|^{N+2s})|y|^{N+2s}}\ dy \leq C\frac{1+|x|^{N+2s}}{|x|^{N+2s}}\int_{\R^N}(1+|x+y|^{N+2s})^{-1}\ dy<C.
 \end{align*}
This implies \eqref{A6} and finishes the proof.
\end{proof}

\begin{cor}\label{boundedevaluation}
For every $f\in {C^{2m+2}_c(\R^N)}$ there exists a constant $C=C(N,m,\sigma,f)>0$ such that $\cE_{s}( f, \varphi)\leq C\int_{\R^N} \varphi(y)\ dy$ for all nonnegative $\varphi\in H^{s}(\R^N)$ and $\|(-\Delta)^{s} f\|_{L^{\infty}(\R^N)}\leq C$.
\end{cor}
\begin{proof}
Note that by Lemma \ref{existence} we have $\cE_{s}( f, \varphi)=\int_{\R^N}(-\Delta)^{s} f(x)\varphi(x)\ dx$. Moreover, since $f\in C^{2m+2}_c(\R^N)$ we have $(-\Delta)^m f\in C^{2}_c(\R^N)$ and thus there is $C>0$ such that (see \emph{\emph{e.g.}} \cite{S07} or using Lemma \ref{boundedevaluationexistence}) $\|(-\Delta)^{s} f\|_{L^{\infty}(\R^N)}\leq C.$ Hence $\cE_{s}( f,\varphi)\leq C\int_{\R^N}\varphi(y)\ dy$ as claimed.
\end{proof}

\begin{lemma}\label{boundedevaluation2}
 Let $U,D\subset \R^N$ open sets with Lipschitz boundary and $\dist(U,D)>0$, $\phi\in \cH^{s}_{0}(U)$, and $g\in \cH^{s}_{0}(D)$. Then there is $C=C(N,m,\sigma)>0$ such that
\[
 \cE_{s}(g,\varphi)=(-1)^{m+1}C\int_{U}\int_{D}\frac{\varphi(x)g(y)}{|x-y|^{N+2s}}\ dxdy.
\]
\end{lemma}
\begin{proof}
 Let $g,\phi$ be as stated. If $m$ is even, we have using Green's formula 
\begin{align*}
 \cE_{s}&(g,\varphi)=-\frac{c_{N,\sigma}}{2}\int_{U}\int_{D}\frac{(-\Delta)^{\frac{m}{2}} \phi(x)(-\Delta)^{\frac{m}{2}} g(y)}{|x-y|^{N+2\sigma}}\ dydx\\
&=-\frac{c_{N,\sigma}}{2}\int_{U} \phi(x)\int_{D}(-\Delta)^{\frac{m}{2}}g(y) (-\Delta)^{\frac{m}{2}}_{x}{|x-y|^{-N-2\sigma}}\ dy dx\\
&= -\frac{c_{N,\sigma}}{2}\int_{U} \phi(x) \int_{D} g(y)(-\Delta)^{\frac{m}{2}}_{y}(-\Delta)^{\frac{m}{2}}_{x}{|x-y|^{-N-2\sigma}}\ dy dx\\
&= -\frac{c_{N,\sigma}}{2}\int_{U} \phi(x) \int_{D} g(y)(-\Delta)^{m}_{y}{|x-y|^{-N-2\sigma}}\ dy dx,
\end{align*}
where we used $(-\Delta)^{\frac{m}{2}}_{y}|x-y|^{-N-2\sigma}=(-\Delta)^{\frac{m}{2}}_{x}|x-y|^{-N-2\sigma}$.\\
If $m$ is odd we have by integration by parts
\begin{align*}
 \cE_{s}&(g,\varphi)=-\frac{c_{N,\sigma}}{2}\int_{U}\int_{D}\frac{\nabla(-\Delta)^{\frac{m-1}{2}} \phi(x)\nabla (-\Delta)^{\frac{m-1}{2}}g(y)}{|x-y|^{N+2\sigma}}\ dydx\\
&=\frac{c_{N,\sigma}}{2}\int_{U} (-\Delta)^{\frac{m-1}{2}}\phi(x)\int_{D}\nabla (-\Delta)^{\frac{m-1}{2}}g(y) \nabla_{x}{|x-y|^{-N-2\sigma}}\ dy dx\\
&=-\frac{c_{N,\sigma}}{2} \int_{U}(-\Delta)^{\frac{m-1}{2}} \phi(x) \int_{D}\nabla (-\Delta)^{\frac{m-1}{2}}g(y)\nabla_{y}{|x-y|^{-N-2\sigma}}\ dy dx\\
&=-\frac{c_{N,\sigma}}{2} \int_{U}(-\Delta)^{\frac{m-1}{2}} \phi(x) \int_{D}(-\Delta)^{\frac{m-1}{2}} g(y)(-\Delta_{y}){|x-y|^{-N-2\sigma}}\ dy dx\\
&=-\frac{c_{N,\sigma}}{2}\int_{U} \phi(x) \int_{D} g(y)(-\Delta_{y})^{m}{|x-y|^{-N-2\sigma}}\ dy dx,
\end{align*}
where the last step follows as in the case $m$ even. Hence to finish the proof, note that for $x\in U$, $y\in D$ and $k>0$ we have $(-\Delta)_{y} |y-x|^{-k}\ dy=k(N-k-2)|y-x|^{-k-2},$ which gives
\begin{align*}
(-\Delta)^m_{y} {|y-x|^{-N-2\sigma}} &=-(N+2\sigma)(2\sigma+2) (-\Delta)_y^{m-1}{|y-x|^{-N-2\sigma-2}}\\
&=(-1)^{m}\prod\limits_{i=0}^{m-1}(N+2\sigma+2i)(2\sigma+2(i+1)){|y-x|^{-N-2\sigma-2m}}.
\end{align*}
\end{proof}

\begin{proof}[Proof of Corollary \ref{smoothgoutside}]
Let $\varphi\in \cH^s_0(\R^N\setminus \Omega)\backslash\{0\}$ be nonnegative. Then, by Lemmas \ref{existence} and \ref{boundedevaluation2},
\begin{align*}
(-1)^{m+1}\int_{\R^N\setminus \Omega} (-\Delta)^sg(x)\, \varphi(x)\ dx=(-1)^{m+1}\cE_s(g,\varphi)= C\int_{\Omega}\int_{\R^N\setminus \Omega}\frac{\varphi(x)g(y)}{|x-y|^{N+2s}}\ dx\, dy > 0.
\end{align*}
Since $\varphi$ is arbitrarily chosen, we obtain that $(-1)^{m+1}(-\Delta)^sg>0$ in $\R^N\setminus \overline{\Omega}$.
\end{proof}

%

\section{Counterexample to general maximum principles}\label{ce:sec}
Using the calculations in \cite[Table 3, p. 549]{D12} (see also \cite[Lemma 2.2]{RS15}, \cite[Corollary 9]{DG2016}, or Remark \ref{reg-it-arg} below) we have the following.

\begin{cor}\label{sol:ball}
 Let $r>0$, $x_0\in \R^N$, $s=m+\sigma$ with $m\in\mathbb N_0$ and $\sigma\in(0,1]$. Then the unique weak solution $\psi_{r,x_0}\in \cH^{s}_0(B_r(x_0))$ of $(-\Delta)^{s}\psi_{r,x_0}=1$ in $B_r(x_0)$ and $\psi_{r,x_0}=0$ on $\R^N\setminus B_{r}(x_0)$ is given for $x\in B_r(x_0)$ by 
\[
 \psi_{r,x_0}(x)=\left\{\begin{aligned} &\gamma_{N,s}(r^2-|x-x_0|^2)^{s},&& \text{ if }\quad  |x-x_0|<r, \\
                                                      &0, && \text{ if } \quad |x-x_0|\geq r,
                                                     \end{aligned}\right. \quad \text{ where }\ \gamma_{N,s}=\frac{\Gamma(\frac{N}{2})4^{-s}}{\Gamma(s+1)\Gamma(\frac{N}{2}+s)}.
\]
\end{cor}

We are now ready to construct the counterexample.
\begin{proof}[Proof of Theorem \ref{main:thm}]
Let  $m\in \N$ be odd, $\sigma\in(0,1)$, $s:=m+\sigma$, $D\subset \mathbb  R^N$ be an open set such that $\R^N\setminus D$ has nonempty interior, $A$ be an open ball compactly contained in the interior of $\R^N\setminus D$. Let $g\in C^{\infty}_c(D)\backslash\{0\}$ be a nonnegative function and let $\psi\in\cH^{s}_0(A)$ be the weak solution given by Corollary \ref{sol:ball}, in particular
$\psi \geq 0$ in $\R^N$ and $\cE_{s} (\psi,\phi)=\int_A \phi\ dx$ for all $\phi \in \cH^{s}_0(A).$

Let $C=C(N,m,\sigma)>0$ be the constant given by Lemma \ref{boundedevaluation2} and let 
\begin{align}
f(x)&:=\ \left\lbrace\begin{aligned}
& a-C\int_D g(y)|x-y|^{-N-2s}\ dy & &\quad  \text{ for }\ x\in A, \\
& aC\int_A \psi(y)|x-y|^{-N-2s}\ dy-(-\Delta)^s g(x) & &\quad \text{ for }\ x\in D,
\end{aligned}\right.
\end{align}
where $a>0$ is chosen large enough such that $f>0$ in $\overline{\Omega}$ where $\Omega:=D\cup {A}$, which is possible by Corollary \ref{boundedevaluation} and because $\operatorname{dist}(D,A)>0$. Let $u(x):=a\psi(x)-g(x)$ for $x\in \R^N$. Clearly $u\in\cH_0^{s}(\Omega)\cap C(\R^N)\cap C^{\infty}(\Omega)$.  

We now show that $u$ is a sign-changing weak solution of 
\begin{align}\label{eq:cex}
(-\Delta)^{s} u=f\geq 0\quad\text{ in }\Omega,\qquad u=0\quad \text{ on }\R^N\setminus\Omega. 
\end{align}
Let $\varphi\in \cH^{s}_0(\Omega)$ with $\varphi\geq 0$. Then $\varphi=\varphi_D+\varphi_A$ for some nonnegative $\varphi_D\in \cH^{s}_0(D)$ and $\varphi_A\in \cH^{s}_0(A)$.  Since $m$ is odd we have
\begin{align*}
 \cE_{s} (u,\phi_D)&=  a\, \cE_{s} (\psi,\phi_D) -\cE_{s} (g,\phi_D)=a\, C\int_{D} \int_{A} \frac{\phi_D(x)\psi(y)}{|x-y|^{N+2s}}\ dy\,dx-\int_{D}(-\Delta)^s g\ \phi_D\ dx,
\end{align*}
by Lemma \ref{boundedevaluation2} and Remark \ref{commutation}. Thus $\cE_{s} (u,\phi_D)=\int_{D}f(x)\phi_D(x)\ dx.$ Analogously,
\begin{align*}
\cE_{s} (u,\phi_A)&=  a\, \cE_{s} (\psi,\phi_A)-\cE_{s} (g,\phi_A)= a \int_{A} \phi_A\ dx-C\int_{A} \int_{D} \frac{\phi_A(x)g(y)}{|x-y|^{N+2s}}\ dy\,dx,
\end{align*}
which yields that $\cE_{s} (u,\phi_A)=\int_{A}f(x)\phi_A(x)\ dx$.  Therefore $\cE_{s} (u,\phi)= \cE_{s} (f,\phi)$ for all $\varphi\in \cH^{s}_0(\Omega)$ and $u$ is a sign-changing weak solution of \eqref{eq:cex} as claimed.
\end{proof}

\begin{remark}\label{negative_interaction}
If $u\in H^{s}(\R^N)$ and $s\in (0,\frac{3}{2})$ then $u^{\pm}\in H^s(\R^N)$, by \cite[Th\'{e}or\`{e}me 1]{M89}. Hence $\cE_s(|u|,|u|)=\cE_s(u,u)+4\cE_s(u^+,u^-),$ where $|\cE_s(u^+,u^-)|<\infty.$ Note that
\[
\cE_s(u^+,u^-)=\left\{\begin{aligned} &- \int_{\R^N}\int_{\R^N} \frac{u^+(x) u^-(y)}{|x-y|^{N+2s}}\ dxdy && \text{ for $s\in(0,1)$,}\\
&0&&\text{ for $s=1$.}\end{aligned}\right.
\]
Therefore, $\cE_s(|u|,|u|)\leq \cE_s(u,u)$ for all $u\in H^s(\R^N)$, $s\in(0,1]$. This fact seems to be crucial for a classical proof of the weak maximum principle. In the case $s\in(1,\frac{3}{2})$ we have
\begin{align*}
  \cE_s(u^+,u^-)= -\int_{\R^N}\int_{\R^N}\frac{\nabla u^+(x)\cdot \nabla u^-(y)}{|x-y|^{N+2s}}\ dxdy.
\end{align*}
Note that Lemma \ref{boundedevaluation2} suggests that $\cE_s(u^+,u^-)$ is nonnegative and, in particular, if $u\not\equiv|u|$ in $\R^N$ then $\cE_s(|u|,|u|)>\cE_s(u,u)>0$. However, a proof of this fact is still missing.
\end{remark}

\section{The fundamental solution in the whole space} \label{sec:greenfund}

In this section we provide an explicit expression for a fundamental solution of $(-\Delta)^s$ in the whole space $\R^N$. We begin by introducing a weaker notion of solution, {\it i.e.,} solutions in the sense of distributions.

Given $s>0$ we denote (see \emph{\emph{e.g.}} \cite{FW16,S07} for $s\in(0,1)$) 
\begin{align*}
 \cL^1_{s}:=\left\{u\in L^1_{loc}(\R^N)\;:\; \|u\|_{\cL^1_{s}}<\infty \right\}, \qquad \|u\|_{\cL^1_{s}}:=\int_{\R^N}\frac{|u(x)|}{1+|x|^{N+2s}}\ dx.
\end{align*}
\begin{remark}\label{temp:dis}\hspace{1em}
\begin{enumerate}
\item Note that $L^p(\R^N)\subset \cL^1_s\subset \cL^1_{s'}$ for all $0 < s\leq s'$ and $p\in[1,\infty]$.
\item If $u\in \cL^1_{s}$ we can identify $(-\Delta)^su$ with a \emph{tempered distribution} in $\cS'$ satisfying that $\pair{(-\Delta)^{s} u}{\phi}=\int_{\R^N}u(x)(-\Delta)^{s}\phi(x)\ dx$ for all $\phi\in \cS,$ by Lemma \ref{boundedevaluationexistence}.  In particular this also yields that $(-\Delta)^su$ is a \emph{distribution} in $\cD':=(C^{\infty}_c(\R^N))'$ and motivates the following notion of solution.
\end{enumerate} 
\end{remark}

\begin{defi}\label{distributionaldefi}
Let $s>0$, $\Omega\subset \R^N$ open and $f\in \cD'$. A function $u \in \cL^1_{s}$ is called a \textit{distributional solution} of \eqref{wsol:def1} if $u\equiv0$ on $\R^N\setminus \Omega$ and
\begin{equation}\label{dist-sense}
\langle (-\Delta)^s u,\phi\rangle=\langle f,\phi\rangle \qquad\text{ for all }\phi\in C^{\infty}_c(\Omega).
 \end{equation}
A function $u\in \cL^{1}_s$ is called \textit{fundamental solution for $(-\Delta)^s$}, if $(-\Delta)^su=\delta_0$ in $\R^N$ in the sense of distributions, i.e. (\ref{dist-sense}) holds with $f=\delta_0$.
\end{defi}

\begin{defi}\label{s-harmonic-defi}
Let $s>0$, $\Omega\subset \R^N$ open. A function $u\in \cL^1_s$ is called \textit{$s$-harmonic in $\Omega$}, if it satisfies $\langle (-\Delta)^s u,\phi\rangle=0$ for all $\phi\in C^{\infty}_c(\Omega).$
\end{defi}

\begin{remark}\label{0solution}\hspace{1em}
 If $u$ is a fundamental solution, then for any $y\in \R^N$ we have $(-\Delta)^s u(\cdot-y)=\delta_y$ in $\R^N$ in the sense of distributions.
\end{remark}

\begin{remark}\label{weaksolution}
If $\Omega\subset \R^N$ has a continuous boundary, then $C^{\infty}_c(\Omega)$ is dense in $\cH^s_0(\Omega)$ (see \emph{\emph{e.g.}} \cite[Theorem 1.4.2.2]{G11}. 
Therefore, if $u\in \cH^{s}_0(\Omega)$ is a distributional solution of \eqref{wsol:def1} and $\partial\Omega$ is continuous, then, by Lemma \ref{existence}, $u$ is a weak solution, see \eqref{wsol:def}.
This holds in particular if $\Omega=\R^N$ since in this case $\cH^{s}_0(\R^N)=H^{s}(\R^N)$.
\end{remark}

\begin{defi}\label{FNs:def}
 For $s>0$, $N\in \N$, and $x\in \R^N\setminus\{0\}$, define
\begin{align*}
 F_{N,s}(x):=\left\{\begin{aligned}
               &\kappa_{N,s} |x|^{2s-N}, && \quad\text{ if $s-\frac{N}{2}\not\in\N_0$;}\\
&\kappa_{N,s} |x|^{2s-N}\ln|x|, && \quad\text{ if $s-\frac{N}{2}\in\N_0$,}
              \end{aligned}
\right.
\end{align*}
where 
\begin{align*}
\kappa_{N,s}:=\left\{\begin{aligned}
                              &\frac{\Gamma(\frac{N}{2}-s)}{4^{s}\pi^{\frac{N}{2}}\Gamma(s)}, && \quad\text{ if $s-\frac{N}{2}\not\in\N_0$;}\\
&\frac{2^{1-2s}\pi^{-\frac{N}{2}}(-1)^{s+1-\frac{N}{2}}}{\Gamma(s-\frac{N}{2}+1)\Gamma(s)}, && \quad\text{ if $s-\frac{N}{2}\in\N_0$.}
              \end{aligned}
\right.
\end{align*}
\end{defi}

In the following we show that $F_{N,s}$ is a fundamental solution for $(-\Delta)^s$ for all $s>0$.
\begin{remark}
The fact that $F_{N,s}$ is a fundamental solution for $(-\Delta)^s$ with $s>0$ is known, see \cite{GGS10,B16,SKM93}.  The proof we present below is new and relies on induction and recurrence formulas.
\end{remark}

\begin{lemma}\label{rightspacefns}
For all $s>0$ and $N\in \N$ we have $F_{N,s}\in \cL^1_s$.
\end{lemma}
\begin{proof}
The claim follows directly from the following estimates.
\begin{align*}
\int_{\R^N}\frac{|x|^{2s-N}}{1+|x|^{2s+N}}\ dx&\leq \int_{B}|x|^{2s-N}\ dx+ \int_{\R^N\setminus B} |x|^{-2s-N}\ dx<\infty,\quad \text{ if }2s<N;\\
\int_{\R^N}\frac{|x|^{2s-N}}{1+|x|^{2s+N}}\ dx&\leq \int_{B}|x|^{2s-N}\ dx+ \int_{\R^N\setminus B} |x|^{-2N}\ dx<\infty,\quad \text{ if }2s\geq N \text{ and }s-\frac{N}{2}\notin \N_0;\\
\int_{\R^N}\frac{|\ln|x|| |x|^{2s-N}}{1+|x|^{2s+N}}\ dx&\leq \int_{B} -\ln|x|\ dx+ \int_{\R^N\setminus B}\frac{\ln|x|}{|x|^{2N}} \ dx<\infty,\quad \text{ if }\text{ $2s\geq N$ and $s-\frac{N}{2}\in \N_0$.}
\end{align*}
\end{proof}

\begin{lemma}\label{Lap:FS}
Let $s>1$. Then \ $-\Delta F_{N,s}=F_{N,s-1}+R_s$ \ in the sense of distributions, where $R_s$ is an $(s-1)$-harmonic polynomial.
\end{lemma}
\begin{proof}
Let $s>1$ and $x\in \R^N\setminus\{0\}$. If $s-\frac{N}{2}\notin\N_0$ then
\begin{align*}
 -\Delta F_{N,s}(x)&=\kappa_{N,s}(2s-N)2(s-1)|x|^{2(s-1)-N}=F_{N,s-1}(x)
\end{align*}
and the claim follows with $R_s\equiv 0$.  If $s=\frac{N}{2}$, then
\[
-\Delta F_{N,\frac{N}{2}}(x)=-\kappa_{N,\frac{N}{2}}(N-2)|x|^{-2}=F_{N,\frac{N}{2}-1}(x).
\]
and the claim follows with $R_{\frac{N}{2}}\equiv 0$. Finally, if $s-\frac{N}{2}\in\N$, then
\begin{align*}
 -\Delta F_{N,s}(x)&=-\kappa_{N,s}(\Delta|x|^{2s-N}\ln|x|+2\nabla|x|^{2s-N}\nabla\ln|x|+|x|^{2s-N}\Delta\ln|x|)\\
 &=\kappa_{N,s-1}|x|^{2s-N-2}\ln|x|+C_2|x|^{2s-N-2}=F_{N,s-1}+C_2|x|^{2s-N-2},
\end{align*}
where $C_2=(2(N-2s)+(2-N))\kappa_{N,s}$.  The claim follows with $R_s(x):=C_2|x|^{2s-N-2}$, since
\begin{align*}
 (-\Delta)^{s-1}|x|^{2s-N-2}=(-\Delta)^{\frac{N}{2}}(-\Delta)^{\frac{2s-N-2}{2}}|x|^{2s-N-2}=(-\Delta)^{\frac{N}{2}}1=0
\end{align*}
in the sense of distributions, by Lemma \ref{ibyp:RN}.
\end{proof}

\begin{thm}\label{fundamentalrn}
Let $s>0$. Then $F_{N,s}$ is a fundamental solution for $(-\Delta)^s$.
\end{thm}
\begin{proof}
We argue by induction on $s>0$. If $s\in(0,1]$ the claim is known, see \emph{e.g.} \cite[Chapter I]{L72}. 
Let $s>1$ and assume that $F_{N,s-1}$ is a fundamental solution for $(-\Delta)^{s-1}$. Then, by Lemma \ref{Lap:FS}, Lemma \ref{ibyp:RN} and Remark \ref{0solution}, $\pair{(-\Delta)^s F_{N,s}}{ \phi}=\pair{(-\Delta)^{s-1} F_{N,s-1}}{ \phi}=\pair{\delta_0}{\phi}$ for all $\phi\in C^{\infty}_c(\R^N),$ that is, $F_{N,s}$ is a fundamental solution for $(-\Delta)^s$.
\end{proof}

\subsection{Distributional solutions in the whole space}

Next we give some integral bounds for $F_{N,s}\ast f$ for suitable $f\in L^p(\R^N)$. Here, as usual, let $\ast$ denote convolution, that is for functions $u,v:\R^N\to\R$ we put $u\ast v(x):= \int_{\R^N}u(x-y)v(y)dy$ for $x\in\R^N$, whenever the right-hand side exists in a suitable sense.

\begin{lemma}\label{finiteconv}
Let $s>0$ with $2s\geq N$. If $f\in L^1(\R^N)$ has compact support, then $F_{N,s}\ast f\in \cL^1_{s}$.
\end{lemma}

\begin{proof}
Let $s$ and $f$ as in the statement and put $K:=\supp\ f$ and $k:=\sup_{y\in K}|y|^{2s-N}$. 

Consider first that $s-\frac{N}{2}\notin\N_0$. Then
\begin{align*}
|F_{N,s}\ast f(x)|&\leq \kappa_{N,s}\int_{\R^N}(|x|+|y|)^{2s-N}|f(y)|\ dy\leq C\|f\|_{L^1(\R^N)}|x|^{2s-N} + C k\|f\|_{L^1(\R^N)}
\end{align*}
for $x\in \R^N$ and for some constant $C>0$ depending only on $N$ and $s$. By Lemma \ref{rightspacefns} we have that $|x|^{2s-N}\in \cL^1_{s}$ and therefore $F_{N,s}\ast f\in \cL^1_{s}$.

Next, let $s-\frac{N}{2}\in\N_0$ and $x\in \R^N$. Let $z=x-y$ and $B_r=B_r(0)$, then 
\begin{align*}
 |\kappa_{N,s}^{-1} F_{N,s}\ast f(x)|\leq \int_{\{|z|< 1\}}|\ln|z|\ f(y)|\ dy + \int_{\{|z|\geq 1\}}|\ln (|z|)|z|^{2s-N}f(y)|\ dy=:f_1(x)+f_2(x).
\end{align*}
thus
\begin{align*}
 \int_{\R^N}\frac{f_1(x)}{1+|x|^{2s+N}}\ dx&\leq \int_{\R^N}\int_{\{|z|< 1\}}-\ln|z|\ dz |f(y)|\ dy=-\int_{\{|z|< 1\}}\ln(|z|)\ dz \|f(y)\|_{L^1(\R^N)}<\infty,\\
 \int_{\R^N}\frac{f_2(x)}{1+|x|^{2s+N}}\ dx&=\int_{K}\int_{\{|z|\geq 1\}}\frac{\ln(|z|)|z|^{2s-N}}{{1+|x+z|^{2s+N}}}\ dz |f(y)|\ dy\leq M\|f(y)\|_{L^1(\R^N)}<\infty,
 \end{align*}
for some $M>0$ depending only on $s$,$N$, and $K$. Thus $f_1,f_2\in\cL^1_{s}$ and this ends the proof.
\end{proof}

In the case where $2s<N$, the function $F_{N,s}$ has a regularizing effect. For this we use the theory of weak-$L^{p}$-spaces. As in \cite[Chapter 4.3]{LL97} we define $L^{p,w}(\R^{N})$, {$p\geq 1$} as the space of measurable functions $f:\R^{N}\to \R$ such that 
\begin{equation}
 \|f\|_{L^{p,w}(\R^{N})}:=\sup_{A\subset\R^{N}, 0<|A|<\infty}|A|^{-\frac{p-1}{p}}\int_{A}|f(x)|\ dx <\infty.
\end{equation}
The space $L^{p,w}(\R^{N})$ equipped with this norm is a Banach space (see \cite[Chapter 1]{G08}). Note that by H\"{o}lder's inequality $L^{p}(\R^{N})\subset L^{p,w}(\R^{N})$ for all $p\geq 1$.

\begin{lemma}[see also Chapter 4.3 \cite{LL97}]\label{lpw}
 Let $0<\lambda<N$. Then $f(x)=|x|^{-\lambda}\in L^{\frac{N}{\lambda},w}(\R^{N})$. In particular, if $2s<N$, then $F_{N,s}\in L^{q,w}(\R^{N})$ for $q=\frac{N}{N-2s}$.
\end{lemma}
\begin{proof}
Fix $\lambda>0$, $r=\frac{q}{q-1}$ (thus we have $\frac{1}{r}+\frac{1}{q}=1$). Since $N>\lambda$ and $r\mapsto r^{-\lambda}$ is a decreasing function we have
\begin{align*}
 \|f\|_{L^{q,w}(\R^{N})}&=\sup_{A\subset\R^{N}, 0<|A|<\infty}|A|^{-\frac{1}{r}}\int_{A}|x|^{-\lambda}\ dx\\
&=\sup_{R>0}\left(N|B|R^{N}\right)^{-1/r}N|B|\int_{0}^{R}m^{-\lambda+N-1}\ dm= \frac{\left(N|B|\right)^{1-r^{-1}}}{N-\lambda} \sup_{R>0}R^{-N/r+N-\lambda}.
\end{align*}

Thus, if $r=\frac{N}{N-\lambda}$ with $q=\frac{r}{r-1}=\frac{N}{\lambda}$ we get $\|f\|_{L^{q,w}(\R^{N})}=\frac{\left(N|B|\right)^{\frac{1}{q}}}{N-\lambda}<\infty$.
\end{proof}

\begin{thm}[see Theorem 1.2.13, p. 21, \cite{G08}]\label{wyineq}
 Let $U\subset\R^{N}$ be any open set and let $g\in L^{p}(U)$, $1\leq p<\infty$, $k\in L^{q,w}(\R^{N})$ and $r,q\in(1,\infty)$ be given such that
$ \frac{1}{p}+\frac{1}{q}=1+\frac{1}{r}.$
Then there is a constant $C=C(N,q,r)>0$ such that
\[
 \left\|k\ast g\right\|_{L^{r}(U)}=\left\|\;\int_{\R^{N}}k(\cdot-y)g(y)\ dy\right\|_{L^{r}(U)}\leq C\|k\|_{L^{q,w}(\R^{N})}\|g\|_{L^{p}(U)}.
\]
\end{thm}

A direct consequence of Lemma \ref{lpw} and Theorem \ref{wyineq} is

\begin{cor}\label{p-reg}
Let $0<s<\frac{N}{2}$, $1\leq p<\frac{N}{2s}$, and $f\in L^{p}(\R^N)$. Then $F_{N,s}\ast f\in L^{\frac{Np}{N-2sp}}(\R^N)$ 
\end{cor}
\begin{proof}
By Lemma \ref{lpw} we have $F_{N,s}\in L^{q,w}(\R^{N})$ for $q=\frac{N}{N-2s}$. The claim follows by Theorem \ref{wyineq} using $p\in [1,\frac{N}{2s})$ and $r=\frac{Np}{N-2sp}$.
\end{proof}

\begin{cor}\label{p-reg2}
  Let $0<s<\frac{N}{2}$ and $f\in L^{p}(\R^N)$, $1\leq p<\infty$ with compact support. Then $F_{N,s}\ast f \in L^{q}(\R^N)$ for every $q\in[\frac{N}{N-2s},\frac{N}{Np-2sp}]$ if $p<\frac{N}{2s}$ and for every $q\in[\frac{N}{N-2s},\infty)$ if $p\geq\frac{N}{2s}$.
 \end{cor}
\begin{proof}
 Since $f$ has compact support, we have by H\"older's inequality that $f\in L^{\tilde p}(\R^N)$ for every $\tilde p\in[1,\min\{p,\frac{N}{2s})\}$. The result follows by Corollary \ref{p-reg}.
\end{proof}

\begin{cor}\label{generalcase}
Let $s>0$ and $f\in L^1(\R^N)$ with compact support, then $u=F_{N,s}\ast f\in \cL^1_{s}$ is a distributional solution of $(-\Delta)^su=f$ in $\R^N$.
\end{cor}
\begin{proof}
By Lemma \ref{finiteconv} or Corollary \ref{p-reg} we have $u\in \cL^1_s$. And, moreover, 
$$
\pair{(-\Delta)^{s}u}{\phi} = \int_{\R^N} f(y)\pair{(-\Delta)^{s}F_{N,s}(\cdot-y)}{\phi}\ dy=\int_{\R^N} f(y) \phi(y)\ dy
$$
for $\varphi\in C^{\infty}_c(\R^N)$ by Theorem \ref{fundamentalrn} and Lemma \ref{ibyp:RN}.
\end{proof}

\begin{thm}\label{ppprn}
Let $m\in\N$, $\sigma\in(0,1]$, $s=m+\sigma$ such that $0<s<\frac{N}{2}$, $f\in L^1(\R^N)$ have compact support, and $u\in \cL^1_{s}$ be a distributional solution of $(-\Delta)^su=f$ in $\R^N$.  Then $u=F_{N,s}\ast f+ P$, where $P$ is a polynomial of degree $n<2s$ for some $n\in\N_0$.  In particular, if $\lim\limits_{|x|\to\infty}u=0$ then $u=F_{N,s}\ast f$ and $\inf_{K}u>0$ for every $K\subset\subset \R^N$ whenever $f\geq 0$ is nonzero.
\end{thm}
\begin{proof}
	By Theorem \ref{fundamentalrn} we have that $u=F_{N,s}\ast f$ is a distributional solution of $(-\Delta)^su=f$.  We now argue as in \cite[Corollary 2.4.3]{G08}. Let $v\in \cL^1_{s}$ be a distributional solution of $(-\Delta)^s v=f$ in $\R^N$. Then, by Remark \ref{temp:dis} we have that $w:= u-v\in \cL^1_{s}\subset \cS'$ and thus $\pair{(-\Delta)^s w}{\phi}=0$ for all $\phi\in C_c^\infty$. Let $\psi\in \cS$ and $(\phi_n)_n\subset C^{\infty}_c(\R^N)$ such that $\phi_n\to \psi$ in $C^{2m+2}$. Then
	\begin{align*}
	\pair{(-\Delta)^s w}{\psi}=\pair{(-\Delta)^s w}{\psi-\phi_n}\leq C \|\psi-\phi_n\|_{C^{2m+2}(\R^N)} \int_{\R^N} \frac{|w(x)|}{1+|x|^{2s+N}}\ dx \to0
	\end{align*}
	as $n\to\infty$, by Lemma \ref{boundedevaluationexistence}. Therefore $(-\Delta)^s w \in \cS'$ and $\pair{(-\Delta)^sw}{\psi}=0$ for all $\psi \in \cS$.  This implies that $w$ is supported in the origin, and then \cite[Corollary 2.4.2]{G08} yields that $w$ is polynomial of degree $n\in \mathbb N$. Since $w\in \cL^1_{s}$ we have that $n<2s$, and the claim follows.
\end{proof}

\begin{remark}
Note that if $s>\frac{N}{p}$, $f\in L^p(\R^N)$ with compact support, then $F_{N,s}\ast f\in C^{s-\frac{N}{p}}(\R^N)$, see for example \cite[Section 4.2, Theorem 2.2, p.155]{M1996} for the case $s-\frac{N}{p}<1$ and the general case follows by differentiation.
\end{remark}

\section{Representation of solutions in the ball}\label{ballcase}

Let $m\in \N_0$, $\sigma\in(0,1]$, $s=m+\sigma$, $N\in \N$ and recall that $d(x):=\dist(x,B)$ for $x\in \R^N$. In this section provide a representation formula for solutions in a ball in terms of a kernel $\cG_s$ given by \emph{Boggio's formula} \eqref{greenball-intro}. We show that $u(x)=\int_B\cG_s(x,y)f(y)\ dy$ for $x\in\R^N$ if and only if $u$ is a solution (in a suitable sense) of $(-\Delta)^su(x)=f$ in $B$ and $u\equiv 0$ on $\R^N\setminus B$. 

A key ingredient in our proofs is the following iteration formula.
\begin{lemma}\label{Lap:green}
If $s>1$ then $-\Delta_x\ \cG_{s}(x,y)=\cG_{s-1}(x,y)-k_{N,s}4(s-1)P_{s-1}(x,y)$ for all $x,y\in B$, $x\neq y$, where
\begin{align}\label{Psm1}
P_{s-1}(x,y)&:=\frac{(1-|x|^2)_+^{s-2}(1-|y|^2)_+^{s-1}(1-|x|^2|y|^2)}{[x,y]^{N}}
\end{align}
for $x,y\in\R^N$, $x\neq y$, and $[x,y]:=\sqrt{|x|^2|y|^2-2x\cdot y +1}$.
\end{lemma}

The proof of Lemma \ref{Lap:green} is done by an elementary\textemdash but lengthy\textemdash direct computation and for the reader's convenience we give a proof in Appendix \ref{recurrence}.

\begin{remark}\label{knowngreen-ball}\hspace{1em}
\begin{enumerate}
\item For $\sigma=\frac{1}{2}$, $N=1$, the substitution $t=\sqrt{v}$ yields
$G_{1,\frac{1}{2}}(x,y)=\frac{1}{\pi}\ln\left(\frac{1-xy+\sqrt{(1-x^2)(1-y^2)}}{|x-y|}\right),$ which agrees with \cite[Theorem 3.1, formula (3.2)]{B16} and for $s\in\mathbb N$, the change of variables $\tilde v=\sqrt{v+1}$ yields
$\cG_s(x,y)=2k_{N,s}|x-y|^{2s-N}\int_1^{p(x,y)} (v^2-1)^{s-1}v^{1-N}\ dv,$ with $p(x,y)=[x,y]|x-y|^{-1}$,
which is another known expression for Boggio's formula, see \cite{GGS10}.
\item By rescaling we have that Theorem \ref{green:thm} holds in balls of radius $r>0$ using $\rho_r(x,y)=(r^2-|x|^2)(r^2-|y|^2)r^{-2}|x-y|^{-2}$ in place of $\rho$ in \eqref{greenball-intro}.
\end{enumerate}
\end{remark}

\begin{remark}\label{Gsbounds}
The following are well-known estimates for $\cG_s$. They do not play an important role in our proofs, but we state them for completeness.  Let $f,g\geq 0$ be functions defined on the same set $D$. We write $f\preceq g$ if there is $c>0$ such that $f(x)\leq cg(x)$ for all $x\in D$. We write $f \simeq g$ if both $f\preceq g$ and $g\preceq f$.  In $\overline{B}\times \overline{B}$ we have
\begin{align*}
 \cG_s(x,y)\simeq \left\{\begin{aligned}
&|x-y|^{2s-N}\min\Big\{1,\frac{d(x)^sd(y)^s}{|x-y|^{2s}}\Big\},&& \quad \text{ if $N>2s$,}\\
&\ln\Big(1+\frac{d(x)^sd(y)^s}{|x-y|^{2s}}\Big),&& \quad \text{ if $N=2s$,}\\
&d(x)^{s-\frac{N}{2}}d(y)^{s-\frac{N}{2}}\min\left\{1,\frac{d(x)^{\frac{N}{2}}d(y)^{\frac{N}{2}}}{|x-y|^{N}}\right\},&& \quad \text{ if $N<2s$.}\\
\end{aligned}\right.
\end{align*}

These type of estimates are known if $s\in\N\cup(0,1)$, see, for example, \cite{CS98,GGS10}. We refer to \cite[Theorem 4.6]{GGS10}, where the case $s\in\N$ is considered, but the proof carries the fractional case $s>1$.
\end{remark}

The following is a useful auxiliary Lemma.
\begin{lemma}\label{s:arg}
 Let $N\in\mathbb N$, $R,s,r>0$, and $\varepsilon\in(0,\min\{N,s\})$.  Then
 \begin{align*}
  R^{2s-N}\int_0^{\frac{r}{R^2}}\frac{t^{s-1}}{(t+1)^{\frac{N}{2}}}\ dt \leq \frac{2}{s} R^{\varepsilon-N}r^{s-\frac{\varepsilon}{2}}.
 \end{align*}
\end{lemma}
\begin{proof}
Let $\delta\in(0,1)$ such that $\varepsilon:=\frac{N\delta}{2}\in(0,\min\{N, s\})$. By a change of variables we have that
 \begin{align*}
  R^{2s-N}\int_0^{\frac{r}{R^2}}\frac{t^{s-1}}{(t+1)^{\frac{N}{2}}}\ dt=R^{-N}\int_0^{r}\frac{t^{s-1}}{(tR^{-2}+1)^{\frac{N}{2}}}\frac{R^{\varepsilon}}{R^{\varepsilon}}\ dt
  =R^{\varepsilon-N}\int_0^{r}\frac{t^{s-1}}{(tR^{\delta-2}+R^\delta)^{\frac{N}{2}}}\ dt.
 \end{align*}
Note that the function $R\mapsto tR^{\delta-2}+R^\delta$ has a unique minimum in $(0,\infty)$ at $R_0=k\sqrt{t}$ with $k=\sqrt{\frac{2-\delta}{\delta}}$. Therefore
\begin{align*}
  R^{\varepsilon-N}\int_0^{r}\frac{t^{s-1}}{(tR^{\delta-2}+R^\delta)^{\frac{N}{2}}}\ dt&\leq
  R^{\varepsilon-N}\int_0^{r}\frac{t^{s-1}}{(tR_0^{\delta-2}+R_0^\delta)^{\frac{N}{2}}}\ dt
  =R^{\varepsilon-N}\int_0^{r}\frac{t^{s-1}}{(t^{\frac{\delta}{2}}(k^{\delta-2}+k^\delta))^{\frac{N}{2}}}\ dt\\
  &\leq R^{\varepsilon-N}\int_0^{r}\frac{t^{s-1-\frac{\varepsilon}{2}}}{k^\varepsilon}\ dt=\frac{k^{-\varepsilon}}{s-\frac{\varepsilon}{2}}R^{\varepsilon-N}r^{s-\frac{\varepsilon}{2}}
  \leq \frac{2}{s} R^{\varepsilon-N}r^{s-\frac{\varepsilon}{2}},
\end{align*}
since $\varepsilon<s$ and  $k^{-\varepsilon}= \frac{\delta^\frac{\varepsilon}{2}}{(2-\delta)^\frac{\varepsilon}{2}}
 \leq \delta^\frac{\varepsilon}{2} \leq \delta^\frac{N\delta}{4}\leq 1,$ because $\delta\in(0,1)$.
\end{proof}

\subsection{Interior and boundary regularity}

\begin{lemma}\label{lpreg}
Let $1\leq p\leq \infty$, $s>0$, $f\in L^p(B)$, and $u$ as in \eqref{Gsu}. There is $C=C(N,s,p)>0$ such that $\|u\|_{L^p(B)}\leq C\|f\|_{L^p(B)}$.
\end{lemma}
\begin{proof}
For $x\in B$ let $\zeta(x):=\int_{B} \cG_s(x,y)dy=\int_{B} \cG_s(y,x)dy>0$. Note that $C:=\|\zeta\|_{L^\infty (B)}<\infty$, by Lemma \ref{s:arg} or by Remark \ref{Gsbounds}.
Hence, the statement holds for $p=\infty$. For $p<\infty$, by Jensen's inequality,
 \begin{align*}
  \|u\|^p_{L^p(B)}&=\int_B\Bigg|\zeta(x)\int_B f(y)\frac{\cG_s(x,y)}{\zeta(x)}\ dy\Bigg|^p\ dx\leq \int_B\zeta(x)^{p}\int_B |f(y)|^p \frac{\cG_s(x,y)}{\zeta(x)}\ dydx\\
  &=\int_B|f(y)|^p\int_B\zeta(x)^{p-1}\cG_s(x,y)\ dxdy\leq C^{p-1}\int_B|f(y)|^p\zeta(y)dy \leq C^{p}\|f\|^p_{L^p(B)}<\infty.
 \end{align*}
\end{proof}

\begin{lemma}\label{propertiesp}
Let $s>1$, $1< p\leq \infty$, $f\in L^p(B)$, and $v(x):=\int_BP_{s-1}(x,y)f(y)\ dy$, $x\in B$. 
If $p>\frac{N}{s}$, then $v\in C^{\infty}(B)$ and for all $\alpha\in \N_0^{N}$ there is $C=C(N,s,\alpha)>0$ 
\begin{equation}\label{der:v}
\|d^{2-s+|\alpha|}\partial^{\alpha}v\|_{L^\infty(B)}\leq C \|f\|_{L^p(B)}.
\end{equation}
\end{lemma}
\begin{proof}
In the following let $C_i=C_i(N,s,p)>0$, $i=1,2,\ldots$ be constants. Let $x,y\in B$, then 
\begin{align}\label{bra:bd}
[x,y]=\sqrt{|x|^2|y|^2-2x\cdot y+1}\geq 1-|x||y|\geq 1-|y|\geq \frac{1}{2}(1-|y|^2),
\end{align}
and therefore $P_{s-1}(x,y)\leq (1-|x|^2)^{s-2} C_1[x,y]^{s-N}$ for $s>1$. Moreover,
\begin{equation}\label{bound-bracket}
[x,y]\geq C_2\Big|y-\frac{x}{|x|}\Big| \quad\text{ for all $x\in B\setminus B_{\frac{3}{4}}(0)$.}
\end{equation}
Indeed, denote $|x|=r$, $\theta=\frac{x}{|x|}$ and note that $[r\theta,y]=|ry-\theta|$ and, for $r>3/4$,
\begin{align*}
|ry-\theta|^2 &= |(r-1)y+y-\theta|^2= (1-r)^2|y|^2-2(1-r)\langle y,y-\theta\rangle+|y-\theta|^2\\
&\geq -2(1-r)\langle \theta,y-\theta\rangle-2(1-r)|y-\theta|^2+|y-\theta|^2 \\
&\geq -2(1-r)|y|+2(1-r)-2(1-r)|y-\theta|^2+|y-\theta|^2 \\
&\geq -2(1-r)|y-\theta|^2+|y-\theta|^2= \frac{|y-\theta|^2}{2},
\end{align*}
which implies (\ref{bound-bracket}). Note that \eqref{bound-bracket} gives that there is $C_3>0$ such that
\begin{equation}\label{bound-bracket2}
\sup_{x\in B} \int_B [x,y]^{s-N}\ dy\leq C_3.
\end{equation}
Next, let $f\in L^p(B)$, $p\in(1,\infty]$, $s>\frac{N}{p}$, and define $v(x)=\int_BP_{s-1}(x,y)f(y)\ dy$  for $x\in B$. Note that for every $\alpha\in \N_0^N$ there is $C=C(\alpha)>0$ such that $|\partial^{\alpha}v(x)|\leq C(\alpha)\|f\|_{L^p(B)}$ for all $x\in B_{\frac{3}{4}}(0)$. Moreover, for $|x|>\frac{3}{4}$ we have with $q=\frac{p}{p-1}$ for $p<\infty$ and $q=1$ for $p=\infty$
\begin{align*}
|v(x)|&\leq (1-|x|^2)^{s-2}\|f\|_{L^p(B)}\Bigg(\int_B(1-|y|^2)^{(s-1)q}(1-|x|^2|y|^2)^{q}[x,y]^{-Nq}\ dy\Bigg)^{\frac{1}{q}}\\
&\leq 2^s(1-|x|^2)^{s-2}\|f\|_{L^p(B)}\Bigg(\int_B[x,y]^{(s-N)q}\ dy\Bigg)^{\frac{1}{q}}\leq C_4(1-|x|^2)^{s-2}\|f\|_{L^p(B)},
\end{align*}
since $(s-N)+\frac{N}{q}=s-\frac{N}{p}>0$ and using  \eqref{bra:bd} and (\ref{bound-bracket2}). Arguing similarly one can obtain \eqref{der:v} for derivatives of order $k$, since terms of the form $(1-|x|^2)^{s-2}[x,y]^{-N-k}$ can be bounded by $(1-|x|^2)^{s-2-k}[x,y]^{-N}$. Thus, proceeding as above, $|\partial^{\alpha} v(x)|\leq C_5 \|f\|_{L^p(B)}(1-|x|^2)^{s-2-|\alpha|}$ for all $i\in\{1,\ldots,N\}$, and the Lemma follows.
\end{proof}

\begin{prop}\label{bd:reg}
Let $1\leq p\leq \infty$, $k\in \R$, $s>0$, $f:B\to\R$ such that $d^kf\in L^{p}(B)$, and $u$ as in \eqref{Gsu}. If $s>k$, then there is $C=C(N,s,k,p)>0$ such that $\|d^{-s}u\|_{L^p(B)}\leq C\|d^kf\|_{L^{p}(B)}.$
\end{prop}
\begin{proof}
	First, note that given $\epsilon>0$ there is $C=C(\epsilon)>0$ such that $\int_{B}|x-y|^{\epsilon-N}d(x)^{-p\frac{\epsilon}{2}}\ dx\leq C$ for all $y\in B$ and $p< \frac{2}{\epsilon}$. In the following let $C_i=C_i(N,s,p,k)>0$, $i=1,2,\ldots$ be constants. First let $1\leq p<\infty$ and fix $0<\epsilon<\min\{1,s-k,\frac{1}{p}\}$. Then, by Lemma \ref{s:arg} and H\"older's inequality,
\begin{align*}
	\|d^{-s}u\|_{L^p(B)}^p&\leq C_1\int_B\Bigg(\int_B|x-y|^{\epsilon-N} d(x)^{-\frac{\epsilon}{2}}d^{s-k-\frac{\epsilon}{2}}(y)d^k(y)|f(y)|\ dy\Bigg)^p\ dx\\
	&\leq C_2\int_B\Bigg(\int_B|x-y|^{\epsilon-N} d(x)^{-\frac{\epsilon}{2}}d^k(y)|f(y)|\ dy  \Bigg)^p\ dx\\
	&\leq C_3\int_B \Bigg(\int_{B}|x-y|^{\epsilon-N}\ dy\Bigg)^{p-1} \Bigg( \int_{B} d(x)^{-\frac{p\epsilon}{2}} |x-y|^{\epsilon-N}d^{kp}(y)|f(y)|^p\ dy\Bigg)\ dx\\
	&\leq C_4\int_B  \int_{B} d(x)^{-p\frac{\epsilon}{2}} |x-y|^{\epsilon-N}d^{kp}(y)|f(y)|^p\ dy\ dx\\
	&=C_5\int_B  d^{kp}(y)|f(y)|^p \int_{B} d(x)^{-p\frac{\epsilon}{2}} |x-y|^{\epsilon-N}\ dx\ dy\leq C_6\|d^kf\|_{L^p(B)}.
\end{align*}
Next let $p=\infty$, $x\in \R^N\backslash\{0\}$. Then
\begin{align*}
 |d^{-s}(x)u(x)|&\leq k_{N,s}\|d^kf\|_{L^\infty(B)}d^{-s}(x)\int_B |x-y|^{2s-N}d^{-k}(y)\int_0^{\frac{(1-|x|^2)(1-|y|^2)}{|x-y|^2}}\frac{t^{s-1}}{(t+1)^{\frac{N}{2}}}\ dtdy\\
 &\leq 2^sk_{N,s}\|d^kf\|_{L^\infty(B)}\int_B |x-y|^{2s-N}d^{s-k}(y)\int_0^{|x-y|^{-2}}\frac{t^{s-1}}{((1-|y|^2)(1-|x|^2)t+1)^{\frac{N}{2}}}\ dtdy\\
  &\leq 2^sk_{N,s}\|d^kf\|_{L^\infty(B)}\int_B |x-y|^{2s-N}\int_0^{|x-y|^{-2}}d^{s-k}(y)\frac{t^{s-1}}{((1-|y|^2)t+1)^{\frac{N}{2}}}\ dtdy.
\end{align*}
Furthermore,

\begin{align*}
\int_B |x-y|^{2s-N}&\int_0^{|x-y|^{-2}}d^{s-k}(y)\frac{t^{s-1}}{((1-|y|^2)t+1)^{\frac{N}{2}}}\ dtdy\\
&\leq \int_B|x-y|^{2s-N}\ dy +\int_B|x-y|^{2s-N}d^{s-k}(y)\int_1^{\max\{|x-y|^{-2},1\}} \frac{t^{s-1}}{((1-|y|^2)t+1)^{\frac{N}{2}}}\ dt \ dy\\
&\leq C_7+\int_B\int_{\min\{1,|x-y|^{2}\}}^{1}d^{s-k}(y) \frac{t^{s-1}}{((1-|y|^2)t+|x-y|^2)^{\frac{N}{2}}}\ dt \ dy\\
&\leq C_7+\int_Bd^{s-k}(y)\int_{0}^{1} \frac{1}{((1-|y|^2)t+|x-y|^2)^{\frac{N}{2}}}\ dt \ dy\\
&\leq C_8+C_8\int_Bd^{s-k-1}(y) \Bigg|((1-|y|^2)t+|x-y|^2)^{1-\frac{N}{2}}\Bigg|_{0}^1\Bigg| \ dy\\
&\leq C_9+C_9\int_Bd^{s-k-1}(y) |x-y|^{2-N} \ dy<\infty.
\end{align*}

Hence the statement also holds for $p=\infty$.
\end{proof}

The following remarks are used in the proof of Theorem \ref{reg-ball-a} below.
\begin{remark}\label{h-spaces}
For $s\in \R$ let $H^s(B)$ and $\cH^s_0(B)$ as in Section \ref{Notation}.  
\begin{enumerate}
  \item For every $s\geq 0$ and $u:\R^N\to\R$ with $u\equiv 0$ in $\R^N\backslash B$, there is $k>0$ such that
 \begin{align}\label{eq:norm}
 k\|u\|^2_{\cH^{s}_0(B)}\leq \|u\|_{H^s(B)}^2+\|d^{-s}u\|_{L^2(B)}^2\leq \frac{1}{k}\|u\|^2_{\cH^{s}_0(B)},
 \end{align}
 see \cite[Section 4.3.2, eq. (7)]{T78}. 
\item

By \cite[Section 5.7.1 page 402]{T78}, the Laplacian with Dirichlet boundary conditions gives an isomorphic mapping from $H^{2+s}(B)$ onto $H^{s}(B)$ for all $-1<s<\infty$, $s\neq -\frac{1}{2}$, and therefore,
\begin{equation}\label{G1-map2}
\cG_1: H^s(B)\to H^{s+2}(B)\quad \text{ for all }s>-1, \ s\neq-\frac{1}{2}.
\end{equation}

\item  Let $(\cH^s_0(B))'$ denote the dual space of $\cH^s_0(B)$. Then, by \cite[Theorem 2.10.5/1]{T78} (see also \cite{MN15}),
\begin{align}\label{d:em}
(\cH^s_0(B))'= H^{-s}(B)\qquad \text{ for } s\in\R,
\end{align}

\end{enumerate}


\end{remark}

\begin{thm}\label{reg-ball-a} Let $s>0$, $f\in C^{\alpha}(B)$ for some $\alpha\in(0,1)$, and $u$ as in \eqref{Gsu}. Then 
\begin{align*}
u\in C^{2s+\alpha}_{loc}(B)\cap C^s_0(B) \cap \cH_0^s(B). 
\end{align*}
\end{thm}
\begin{proof}
For $s\in\N\cup (0,1)$ the result is known, see \cite[Section 4.2.1]{GGS10} and \cite{S07,nicola,GT,B16,G15:2}. We argue by induction on $s$.  Let $s>1$, $s\not\in N$, and consider the case $2\sigma+\alpha\in(0,1)$ (the other cases can be proved similarly).  By the induction hypothesis, we have that $\cG_{s-1}(\cdot,y), P_{s-1}(\cdot,y)\in L^{1}(B)$ and, by Lemma~\ref{Lap:green},
\begin{equation}\label{eq:composedgreen}
\cG_s(x,y)=\int_B \cG_1(x,z)\cG_{s-1}(z,y)\;dz -C\int_B \cG_1(x,z)P_{s-1}(z,y)\;dz\qquad  \text{ for } x,y\in B
\end{equation}
 with $C=4k_{N,s}(s-1)$. If $u$ is given by \eqref{Gsu}, then \eqref{eq:composedgreen} implies that $u=u_1-C u_2$, where
 \begin{align*}
  u_1(x)&:=\int_B\cG_1(x,z)v_1(z)\;dz,\quad 
  v_1(z):=\int_B \cG_{s-1}(z,y)\,f(y)\;dy,\\
  u_2(x)&:=\int_B \cG_1(x,z)v_2(z)\; dz,\quad 
  v_2(z):=\int_B P_{s-1}(z,y)\,f(y)\;dy.
 \end{align*}
Then $v_1\ \in\ C^{2s-2+\alpha}_{loc}(B)$, by the induction hypothesis, and then $u_1\in\ C^{2s+\alpha}_{loc}(B)$, by classical elliptic regularity. Furthermore, $v_2\in\ C^\infty(B)$, by Lemma \ref{propertiesp}, and thus $u_2 \in\ C^{\infty}(B).$ Therefore $u\in C^{2s+\alpha}_{loc}(B)$ and $u\in C^s_0(B)$, by Proposition \ref{bd:reg}.

It remains to show that $u\in \cH_0^s(B)$. By \eqref{eq:norm} and Proposition \ref{bd:reg}, it suffices to show that $u\in H^s(B)$. Since $v_1\in \cH_0^{s-1}(B)\subset H^{s-1}(B)$, by the induction hypothesis, we obtain that $u_1\in H^{s+1}(B)\subset H^{s}(B)$. 

We now show that $u_2\in H^s(B)$ arguing differently according to the value of $s$.

Assume first that $1< s < \frac{3}{2}$. Then there is $C>0$ such that 
\begin{equation}\label{lf}
 \int_B v_2(x) \varphi(x) \, dx\leq C \int_B (1-|x|^2)^{s-2}\varphi(x)\, dx\leq C\|d^{-(2-s)}\varphi\|_{L^2(B)}\leq C\|\varphi\|_{\cH_0^{2-s}(B)}
\end{equation}
for $\varphi\in \cH^{2-s}_0(B)$, by \eqref{eq:norm}. Then the functional $\cH^{2-s}_0(B)\ni\varphi\mapsto \int_B v_2\varphi\, dx$ is linear and bounded. Therefore, $v_2\in (\cH_0^{2-s}(B))'=H^{s-2}(B)$, by \eqref{d:em}, and thus $u_2\in H^s(B)$, by \eqref{G1-map2}.

Now, let $s=\frac{3}{2}$ and fix $p\in(\frac{2N}{N+1},2)$. Then $v_2\in L^p(B)$ and thus $u_2\in W^{2,p}(B)\subset H^s(B)$, by Sobolev embeddings (see e.g. \cite[Section 4.6.1]{T78}) and \eqref{G1-map2}.

Furthermore, if $2>s>\frac{3}{2}$, then Lemma \ref{propertiesp} implies that $v_2\in L^2(\R^N)$ and then $u_2\in H^2(B)\subset H^s(B)$, by \eqref{G1-map2} and Sobolev embeddings.

For $s=m+\sigma>2$ with  $\sigma \leq \frac{1}{2}$, fix
\begin{align}\label{pqdef}
q:=(1-\frac{\sigma}{2})^{-1}\qquad \text{ and }\qquad p:=\frac{2-2\sigma}{1-\sigma(2-\sigma)}.
\end{align}
Then, by Lemma \ref{propertiesp} and complex interpolation (see \cite[Proposition 2.4]{Lototsky}),
\begin{align*}
 v_2\in W^{m-2,p}(B)\cap W^{m-1,q}(B)\subset[W^{m-2,p}(B)\, ,\, W^{m-1,q}(B)]_{\sigma}=H^{s-2}(B).
\end{align*}
Therefore $v_2\in H^{s-2}(B)$ for all $s>2$, which yields $u_2\in H^s(B)$, by \eqref{G1-map2}.

Finally, if $s=m+\sigma>2$ and $\sigma>\frac{1}{2}$, then $v_2\in H^{m-1}(B)\subset H^{s-2}(B)$, by Lemma \ref{propertiesp}. But then $u_2\in H^s(B)$, by \eqref{G1-map2}, also in this case and the proof is finished. 
\end{proof}

\begin{remark}\label{reg-it-arg} {\ }
\begin{enumerate}
\item If $u_s:=\int_B\cG_{s}(\cdot,y)f(y)\ dy\in H^{s}(B),$ whenever $f\in L^p(B)$, $p>\frac{N}{s}$, and $s\in(0,1)$, then Theorem \ref{reg-ball-a} would also hold for $f\in L^p(B)$ with $p>\frac{N}{s}$ with a very similar proof. 
\item Arguing as in the proof of Theorem \ref{reg-ball-a} one can show that $u_s(x):=(1-|x|^2)_+^s$, $x\in \R^N$, belongs to $\cH^s_0(B)$. Indeed, for $m\in\N_0$, $\sigma\in(0,1]$, and $s=m+\sigma$, we have that $u_s\in H^{m+1}(B)\subset H^s(B)$ if $\sigma>\frac{1}{2}$ and $u_s\in W^{m,p}(B)\cap W^{m+1,q}(B)\subset H^s(B)$ if $\sigma\leq \frac{1}{2}$, where $p$ and $q$ are as in \eqref{pqdef}. But then $u_s\in \cH_0^{s}(B)$, by \eqref{eq:norm}.
\end{enumerate}	
	
\end{remark}

\subsection{Remarks on \texorpdfstring{$s$}{s}-harmonic functions}

For $s>0$ we define $M_s$ the \emph{$s$-Martin kernel for the ball} by (see for example \cite{B99,nicola})
\begin{align*}
{M_s}(x,\theta):=\lim_{z\to\theta,z\in B}\frac{\cG_{s}(x,z)}{(1-|z|^2)^{s}}\qquad \text{ for } x\in B,\ \theta\in \partial B.
\end{align*}

The next Lemma provides an explicit formula for $M_s$.

\begin{lemma}\label{Martin:ball}
Let $s>0$ and $N\geq 1$. Then
\begin{align*}
 M_s(x,\theta)=\frac{k_{N,s}}{s}\frac{(1-|x|^2)^s_{ +}}{|\theta-x|^N}\qquad \text{ for } x\in B,\ \theta\in \partial B,
\end{align*}
where $k_{N,s}$ is as in \eqref{greenconst-intro}.
\end{lemma}
\begin{proof}
For $x,z\in \R^N$ with $x\neq z$ and $\rho(x,z)=(1-|x|^2)_+(1-|z|^2)_+|x-z|^{-2}$ let $t=\rho(x,z)$, then
\[
\cG_s(x,z)=k_{N,s}(1-|x|^2)^s_+(1-|z|^2)^s_+\int_0^1 \frac{t^{s-1}}{((1-|x|^2)_+(1-|z|^2)_+ t +|x-z|^2)^{\frac{N}{2}}}\ dt.
\]	
Hence, for $\theta\in\partial B$ and $x\in B$, it follows that
\begin{align*}
M_s(x,\theta)&=k_{N,s}(1-|x|^2)^s\lim_{z\to\theta,z\in B}\int_0^1 \frac{t^{s-1}}{((1-|x|^2)_+(1-|z|^2)_+ t +|x-z|^2)^{\frac{N}{2}}}\ dt\\
&=k_{N,s}\frac{(1-|x|^2)^s}{|x-\theta|^{N}}\int_0^1 t^{s-1}\ dt=\frac{k_{N,s}}{s}\frac{(1-|x|^2)^s}{|x-\theta|^{N}}
\end{align*}	
\end{proof}

Martin kernels provide a useful characterization of some $s$-harmonic functions. 
\begin{lemma}\label{lem:sharm}
Let $s>0$ and assume
	\begin{equation}\label{eq:point_s-1}
	\int_B\cG_{s}(x,y)(-\Delta)^{s}\psi(y)\ dy=\psi(x) \quad\text{ for all $x\in B$ and  $\psi\in C^{\infty}_c(B)$.}
	\end{equation}
If  $\mu\in\cM(\partial B)$ is a finite Radon measure, then the function $\R^N\ni x\mapsto u(x):=\int_{\partial B}M_{s}(x,z)\;d\mu(z)$ is $s$-harmonic in $B$.
\end{lemma}

\begin{proof}

We first show that $u\in L^1(B)$. Indeed,
\[
\int_B |u(x)|\;dx\leq\int_{\partial B}\int_B M_{s}(x,z)\;dx\;d|\mu|(z)
\leq 2^sk_{N,s}\int_{\partial B}\int_B |x-z|^{s-N}\;dxd|\mu|(z)<+\infty.
\]
Since $u=0$ in $\R^N\setminus B$, then $u\in\cL^1_{s}$.  Let $\psi\in C_c^\infty(B)$ and note that $u\in C^{\infty}(B)$. Then $(-\Delta)^{s} u(x)$ exists for all $x\in B$ and, by \ref{eq:point_s-1},
\begin{align*}
\langle {(-\Delta)^{s}u},\psi\rangle &=\int_{B}{u(x)}(-\Delta)^{s}\psi(x)\ dx=\int_{B}\int_{\partial B} M_{s}(x,\theta)\ d\mu(\theta) (-\Delta)^{s}\psi(x)\ dx\\
 &=\int_{B}\int_{\partial B} \lim_{z\to\theta,z\in B}\frac{\cG_{s}(x,z)}{(1-|z|^2)^{{s}}}\ d\mu(\theta)(-\Delta)^{s}\psi(x)\ dx\\
 &=\int_{\partial B}\lim_{z\to\theta,z\in B}\frac{1}{(1-|z|^2)^{{s}}}\int_{B} \cG_{s}(x,z) (-\Delta)^{s}\psi(x)\ dx\ d\mu(\theta)\\
 &=\int_{\partial B}\lim_{z\to\theta,z\in B}\frac{\psi(z)}{(1-|z|^2)^{{s}}}\ d\mu(\theta)=0,
\end{align*}
since $\psi$ has compact support in $B$. Therefore $u$ is $s$-harmonic.
\end{proof}

\begin{remark}
We assume \eqref{eq:point_s-1} as part of our iteration argument, but once Theorem \ref{green:thm} is proved then \eqref{eq:point_s-1} holds for all $s>0$.
\end{remark}

%
%
%
%

We now show the relationship between $P_{s-1}$ from Lemma \ref{Lap:green} and $M_s$.
\begin{lemma}\label{Martin:rep} 
Let $s>0$, and $y\in B$. Then 
\[
P_{s-1}(x,y)=\frac{2k_{N,1}(s-1)s}{k_{N,s-1}k_{N,s}}\int_{\partial B}M_{s-1}(x,\theta)M_{s}(y,\theta)\ d\theta\quad\text{ for $x\in B$.}
\]
\end{lemma}
\begin{proof}
Fix $y\in B$ and let $v(x):=\frac{(1-|x|^2|y|^2)}{(1-|y|^2)[x,y]^{N}}$ for $x\in B$. Note that $-\Delta v=0$ in $B$ and $v(\theta)=|\theta-y|^{-N}$ for $\theta\in\partial B.$  Indeed, if $y=0$ then $v\equiv 1$ and if $y\in B\backslash\{0\}$ then $v(x)=\frac{|\eta|^{N}}{|\eta|^2-1}\frac{|\eta|^2-|x|^2}{|x-\eta|^N}$ with $\eta:=\frac{y}{|y|^2}$, and $-\Delta v=0$ follows by a simple calculation. 
Then, by uniqueness and using the Poisson kernel for the Laplacian, 
\begin{align*}
\frac{(1-|x|^2|y|^2)}{(1-|y|^2)[x,y]^{N}}=v(x)=2k_{N,1}\int_{\partial B}\frac{1-|x|^2}{|x-\theta|^N[\theta,y]^{N}}\ d\theta.
\end{align*}
Therefore,
\begin{align*}
P_{s-1}&(x,y)=(1-|x|^2)^{s-2}(1-|y|^2)^{s}\frac{(1-|x|^2|y|^2)}{(1-|y|^2)[x,y]^{N}}\\
&={2k_{N,1}}(1-|x|^2)^{s-2}(1-|y|^2)^{s}\int_{\partial B}\frac{1-|x|^2}{|x-\theta|^N[\theta,y]^{N}}\ d\theta ={2k_{N,1}}\int_{\partial B}\frac{(1-|x|^2)^{s-1}}{|x-\theta|^N}\frac{(1-|y|^2)^{s}}{|\theta-y|^{N}}\ d\theta
\\&={\frac{2k_{N,1}(s-1)s}{k_{N,s-1}{k_{N,s}}}}\int_{\partial B}M_{s-1}(x,\theta){M_{s}(y,\theta)}\ d\theta,
\end{align*}
by Lemma \ref{Martin:ball}, as claimed. 
\end{proof}

\begin{cor}\label{Q:harmonic}
Let $y\in B$ and $s>1$.  If \eqref{eq:point_s-1} holds, then $P_{s-1}(\cdot,y)$ is $(s-1)$-harmonic in $B$.
\end{cor}

\begin{proof}
 Combine Lemma \ref{Martin:rep} and Lemma \ref{lem:sharm}.
\end{proof}

\begin{remark}\label{rmk:sharm}   {\ }
\begin{enumerate}
 \item As mentioned before, the Martin kernel $M_s$ provides a useful characterization of some $s$-harmonic functions. This characterization is new for $s>1$ and may be of independent interest. Namely, if $s>0$ and $g\in C(\partial B)$, then $v(x):=\int_{\partial B}M_s(x,\theta)g(\theta)\ d\theta$ for $x\in B$, is $s$-harmonic.
\item Arguing as in \cite{nicola}, it is possible to prove that if $g\in C(\partial{B})$, then 
 \begin{align*}
  \lim_{z\to\tilde \theta,z\in B}\frac{\int_{\partial B} M_s(z,\theta)g(\theta)\ d\theta}{(1-|z|^2)^{s-1}}={\frac{k_{N,s}}{2k_{N,1}{s}} }g(\tilde\theta)\qquad \text{ for $\tilde \theta\in\partial B$.}
 \end{align*}
Therefore, if $v=\int_{\partial B}M_s(\cdot,\theta)g(\theta)\ d\theta$, then $g(\theta)=2k_{N,1} k^{-1}_{N,s}{s}\lim\limits_{z\to \theta,\ z\in B}v(z)(1-|z|^2)^{1-s}.$
\item If $\varphi\in C^2(B) \cap C(\overline B)$ is harmonic, i.e. $-\Delta\varphi=0$ in $B$, then $u(x):=(1-|x|^2)_+^{s-1}\varphi(x)$, $x\in \R^N$ is $s$-harmonic in $B$. Indeed, using the Poisson kernel representation and Lemma \ref{Martin:ball} we have that 
 \begin{align*}
  u(x)=2k_{N,1}(1-|x|^2)^{s-1}\int_{\partial B}\frac{1-|x|^2}{|x-\theta|^N}\varphi(\theta)\ d\theta=\frac{ 2k_{N,1}{s}}{k_{N,s}}\int_{\partial B}M_s(x,\theta)\varphi(\theta)\ d\theta,
 \end{align*}
 and then $(-\Delta)^s u = 0$ in $B$, by the first Remark.
 \item If a function $u$ is $s$-harmonic in $B$, then $u$ is $(s+1)$-harmonic. Indeed, $\int_{\R^N} u{(-\Delta)}^{s+1}\phi\ dx=\int_{\R^N} u{(-\Delta)}^s[-\Delta\phi]\ dx=0$ for any $\phi\in C^\infty_c(B).$  Thus, for $j\in(0,s)\cap\N$  functions of the type $\int_{\partial B} M_{s-j}(x,\theta) g(\theta)\ d\theta$ are also $s$-harmonic.
\end{enumerate}
\end{remark}

\subsection{Proof of Theorem \ref{green:thm} and consequences}

Recall the \emph{dual pairing} notation $\pair{\cdot}{\cdot}$ introduced in Section \ref{Notation} (see also Section \ref{sec:greenfund}).

\begin{proof}[Proof of Theorem \ref{green:thm}]
Let $f\in C^{\alpha}(B)$ for some $\alpha\in(0,1)$ with $2s+\alpha\not\in \N$ and $u$ as in \eqref{Gsu}. The claim is known for $s\in(0,1]$, see \cite{GGS10,BGR61, B16}. Let $s>1$ and assume that the statement holds for $s-1$. Then $u\in C^{2s+\alpha}_{loc}(B)\cap C^s_0(B)\cap\cH^s_0(B)$, by Theorem \ref{reg-ball-a}. Furthermore, by Lemmas \ref{Lap:green}, \ref{Q:harmonic}, \ref{ibyp}, and the induction hypothesis,
\begin{align*}
&\pair{(-\Delta)^su}{\phi}=\int_B u (-\Delta)^s\phi\ dx= \int_{B}-\Delta u (-\Delta)^{s-1}\phi \ dx\\
&=\pair{\int_B \cG_{s-1}(\cdot,y)f(y)\;dy}{(-\Delta)^{s-1} \phi}-4k_{N,s}(s-1)\int_B f(y)\pair{P_{s-1}(\cdot,y)}{(-\Delta)^{s-1} \phi}\;dy=\pair{f}{\phi}
\end{align*}
for all $\phi\in C^{\infty}_c(B)$, in particular, 
\[
\int_B\int_B\cG_s(x,y)(-\Delta)^s\phi(y)\ dy\ f(x)\ dx=\int_B u(x)(-\Delta)^s\phi(x)\ dx=\int_Bf(x)\phi(x)\ dx.
\]
for any $\phi\in C^{\infty}_c(B)$. Since $f\in C^{\alpha}(B)$ is arbitrary, we obtain that $\int_B\cG_s(x,y)(-\Delta)^s\phi(y)\ dy=\phi(x)$ for every $x\in B$ and thus $\cG_s(\cdot,y)$ is a distributional solution of $(-\Delta)^sv=\delta_y$.  Finally, $u$ is the unique weak solution of \eqref{wsol:def1} with $\Omega=B$ and satisfies $(-\Delta)^{m}(-\Delta)^{\sigma}u(x)=f(x)$ pointwise for every $x\in B$, by Lemmas \ref{ibyp} and \ref{unique:weak} (see also Remark \ref{weaksolution}) and the decay \eqref{decay} follows from Proposition \ref{bd:reg}.
\end{proof}

\begin{proof}[Proof of Corollary \ref{green:cor2}]
Let $j\in\N$ and $s>j$. For any $\phi\in C^\infty_c(B)$ we have that $(-\Delta)^j\phi\in C^\infty_c(B)$ and thus, for $x\in B$,
\[
\int_B \cG_{s-j}(x,y){(-\Delta)}^s\phi(y)\ dy\ =\ \int_B \cG_{s-j}(x,y)(-\Delta)^{s-j}(-\Delta)^j\phi(y)\ dy\ =\ {(-\Delta)}^j\phi(x),
\]
by Proposition \ref{interchange-der} and Theorem \ref{green:thm}, using that $(-\Delta)^{s-j}v=(-\Delta)^j\phi$ in $B$ has a unique solution in $\cH^{s-j}_0(B)$. Let $\mu$ be a finite Radon measure and $u_j=\int_B \cG_{s-j}(\cdot,y)\int_B \cG_j(y,z)\ d\mu(z) dy$, then  
\begin{align*}
\int_B u_j {(-\Delta)}^s\phi\ dx&= \int_B\int_B\cG_{s-j}(x,y)\int_B \cG_j(y,z)\ d\mu(z)\ dy\ {(-\Delta)}^s\phi(x)\ dx \\
&= \int_B\int_B\cG_j(y,z)\int_B \cG_{s-j}(x,y){(-\Delta)}^{s}\phi(x) \ dx\ dy\ d\mu(z)\\
&= \int_B\int_B\cG_j(y,z)(-\Delta)^{j}\phi(y) \ dy\ d\mu(z)= \int_B \phi(z) \ d\mu(z).
\end{align*}
In particular, if $d\mu(z) = f(z)\ dz$ for some $f\in C^{\alpha}(B)$, then, by Theorem \ref{reg-ball-a}, 
\begin{equation*}
 y\mapsto\int_B \cG_j(y,z)\,f(z)\ dz\in C^\alpha(B)\quad \text{ and }\quad  x\mapsto \int_B \cG_{s-j}(x,y)\int_B \cG_j(y,z)\,f(z)\ dzdy \in C^{s-j}_0(B).
\end{equation*}
\end{proof}

\begin{proof}[Proof of Corollary \ref{green:cor}]
Let $v$ as in the statement, fix $y\in B$, and let $\mu=\delta_y$ be a Dirac measure centered at $y$. Then, by Corollary \ref{green:cor2} and Theorem \ref{green:thm}, $\int_B\cG_1(\cdot,z)\cG_{s-1}(z,y)\ dz$ and $\cG_s(\cdot,y)$ are two distributional solutions of $(-\Delta)^s w = \delta_y$, and therefore
$\pair{(-\Delta)^s v}{\varphi}=0$, \emph{i.e.}, $v$ is $s$-harmonic with respect to $x$ in $B$.  Next, fix $x\in B$ and recall formula \eqref{eq:composedgreen}. By Lemma \ref{Martin:rep}, we have that
\begin{align*}
\int_B \cG_1(x,z)\,P_{s-1}(z,y)\ dz &=\frac{2k_{N,1}{(s-1)}}{k_{N,s-1}}
\int_{\partial B}\frac{(1-|y|^2)^{s}}{\left|\theta-y\right|^N}\int_B \cG_1(x,z)\,M_{s-1}(z,\theta)\,\ dz
\ \ d\theta \\
&=\frac{2k_{N,1}{(s-1)s}}{k_{N,s-1}k_{N,s}}
\int_{\partial B}M_{s}(y,\theta)\int_B \cG_1(x,z)\,M_{s-1}(z,\theta)\,\ dz\ d\theta,
\end{align*}
which, by Lemma \ref{lem:sharm}, yields that $v$ is $s$-harmonic with respect to $y$ in $B$.
\end{proof}

\begin{proof}[Proof of Corollary \ref{harm:cor}] Let $g\in C^\infty_c(\R^N\setminus\overline{B})$ with $g\geq 0$. The existence and uniqueness of a weak solution $u\in H^s(\R^N)$ to the problem $(-\Delta)^s u=0$ in $B$ with 
	$u=g$ in $\R^N\setminus B$ follows from standard arguments by minimizing $\cE_s(v,v)$ among all $v\in H^s(\R^N)$ such that $v-g\in\cH^s_0(B)$.  Then $u=g+w$ for some $w\in\cH^s_0(B)$. Moreover, by Lemma \ref{boundedevaluation2}, for all $\phi\in \cH^s_0(B),$ $\phi\geq 0$,
	\begin{align*}
	\cE_s(w,\phi)=\cE_s(u,\phi)-\cE_s(g,\phi)=-\cE_s(g,\phi)=-\int_{B}\tilde{g}(x)\phi(x)\ dx\leq 0,
	\end{align*}
	where $\tilde{g}$ is a smooth function given by
	\begin{align*}
	\tilde{g}(x):=C\int_{\R^{N}\backslash B} \frac{g(y)}{|x-y|^{N+2s}}dy\geq 0\quad \text{ for }x\in B,
	\end{align*}
	for some $C>0$.  In particular, $w\leq 0$ in $\R^N$, by Theorem \ref{green:thm}, and therefore $u\leq 0$ in $B$.
\end{proof}

\appendix

\section{Differential recurrence equation}\label{recurrence}

\begin{proof}[Proof of Lemma \ref{Lap:green}].
Let $s>1$, $y\in B$, $x\in \R^N$, and $x\neq y$, and $\rho$ as in \ref{greenconst-intro}. In the following, differentiation is always w.r.t. $x$. To simplify notation we write $F_s:= |x-y|^{2s-N}$ and $V_s(v):=v^{s-1}(v+1)^{-\frac{N}{2}}$.  

We consider first the case $2s\neq N$. Note that 
\begin{align*}
\nabla F_s=(2s-N)F_{s-1} (x-y)=(2s-N)F_s\frac{x-y}{|x-y|^2}\quad \text{ and }\quad -\Delta\ F_s= (N-2s)2(s-1)F_{s-1},  
\end{align*}
hence
\begin{align}
-\Delta& \cG_{s}(x,y)=-k_{N,s}(\Delta F_s \int_{0}^{\rho}V_s(v)\ dv+2 V_s(\rho)\nabla F_s\cdot\nabla \rho +V_s'(\rho)F_s |\nabla \rho|^2 +F_sV_s(\rho)\Delta \rho ).\label{all}
\end{align}
Note that, for $a\geq 0$,
\begin{align}\label{eq:1}
\int_0^{a} V_s(v)\ dv=\frac{2}{2s-N}\frac{a^{s-1}}{(a+1)^{\frac{N}{2}-1}}-\frac{2(s-1)}{2s-N}\int_0^{a} V_{s-1}(v)\ dv.
\end{align}
Thus, using \eqref{eq:1}, we obtain
\begin{align*}
&-k_{N,s}\Delta F_s \int_0^{\rho}V_s(v)\ dv=\cG_{s-1}(x,y)-k_{N,s}4(s-1)\frac{F_{s}}{|x-y|^2} \frac{\rho^{s-1}}{(\rho+1)^{\frac{N}{2}-1}}.
\end{align*}

Then, $-\Delta\ \cG_{s}=\cG_{s-1}-k_{N,s}4(s-1)P$, where
\begin{align*}
P:=\frac{F_s}{|x-y|^2}  \frac{\rho^{s-1}}{(\rho+1)^{\frac{N}{2}-1}}+\frac{2 V_s(\rho)\nabla F_s\cdot\nabla \rho +F_sV_s'(\rho) |\nabla \rho|^2 +F_sV_s(\rho)\Delta \rho}{4(s-1)}.
\end{align*}
It suffices to show that $P=P_{s-1}$, with $P_{s-1}$ given by \eqref{Psm1}. Note that
\begin{align}
&4(s-1) P=4(s-1)\frac{F_s}{|x-y|^2}  \frac{\rho^{s-1}}{(\rho+1)^{\frac{N}{2}-1}}+2 V_s(\rho)\nabla F_s\cdot\nabla \rho +F_sV_s'(\rho) |\nabla \rho|^2 +F_sV_s(\rho)\Delta \rho\notag\\
&=F_s\Big[\frac{V_s(\rho)(4(s-1)(\rho+1)+2(2s-N)(x-y)\cdot \nabla \rho+|x-y|^2\Delta \rho)}{|x-y|^2}+V_s'(\rho) |\nabla \rho|^2 \Big].\label{harmonic-q1}
\end{align}

To simplify this expression we use
\begin{align*}
V_s'(v)&=(s-1)\frac{v^{s-2}}{(v+1)^{\frac{N}{2}}}-\frac{N}{2}\frac{v^{s-1}}{(v+1)^{\frac{N}{2}+1}}=V_s(v)\frac{(s-1)(v+1)-\frac{N}{2}v}{v(v+1)}
\end{align*}
so that
\begin{align}
&4(s-1)P=F_sV_s(\rho)\Big[\frac{4(s-1)(\rho+1)+2(2s-N)(x-y)\cdot \nabla \rho+|x-y|^2\Delta \rho}{|x-y|^2}\notag\\
&\qquad\qquad\qquad\qquad\qquad\qquad +\frac{(s-1)(\rho+1)-\frac{N}{2}\rho}{\rho(\rho+1)}|\nabla \rho|^2 \Big]\notag\\
&=F_{s-1}V_s(\rho)\Big[\frac{4(s-1)((1-|x|^2)(1-|y|^2)+|x-y|^2)}{|x-y|^2}+2(2s-N)(x-y)\cdot \nabla \rho+|x-y|^2\Delta \rho\notag\\
&\qquad\qquad\qquad  +\frac{(s-1-\frac{N}{2})(1-|x|^2)(1-|y|^2)+(s-1)|x-y|^2}{(1-|x|^2)^2(1-|y|^2)^2+(1-|x|^2)(1-|y|^2)|x-y|^2)}|x-y|^4|\nabla \rho|^2\Big]\label{harmonic-q2}
\end{align}

Direct calculations yield that
\begin{align*}
\Delta \rho
&=\frac{2(1-|y|^2)}{|x-y|^4}\left(-N(|y|^2-2x\cdot y+1)+4(1-x\cdot y)\right)\\
(x-y)\cdot \nabla \rho
&=-2\frac{1-|y|^2}{|x-y|^2} (|x|^2-x\cdot y +1-|x|^2)=-2\frac{(1-|y|^2)(1-x\cdot y)}{|x-y|^2}
\end{align*}
Hence the first three terms in (\ref{harmonic-q2}) reduce to
\begin{align}
\frac{4}{|x-y|^2}[(s-1)(1-2x\cdot y+|x|^2|y|^2)-(1-|y|^2)(\frac{N}{2}(|y|^2-2x\cdot y+1)+(2s-2-N)(1-x\cdot y))]\label{q2-part1}
\end{align}
and the last term in (\ref{harmonic-q2}) reduce to
\begin{align}
4(1-|y|^2)\frac{(s-1-\frac{N}{2})(1-|x|^2)(1-|y|^2)+(s-1)|x-y|^2}{(1-|x|^2)|x-y|^2}. \label{q2-part2}
\end{align}
Combining \eqref{q2-part1}, \eqref{q2-part2} with \eqref{harmonic-q2} we find
\begin{align}
&4(s-1)P=\frac{4F_{s-1}V_s(\rho)}{(1-|x|^2)|x-y|^2}\Big[(s-1)(1-2x\cdot y+|x|^2|y|^2)(1-|x|^2)\notag\\
&\qquad\qquad+(1-|y|^2)\Big(-\frac{N}{2}(|y|^2-2x\cdot y+1)(1-|x|^2)+ (s-1)|x-y|^2\notag\\
&\qquad\qquad\qquad -(2s-2-N)(1-x\cdot y)(1-|x|^2)+ (s-1-\frac{N}{2})(1-|y|^2)(1-|x|^2)\Big)\Big].\label{harmonic-q3}
\end{align}
Note that the bracket in (\ref{harmonic-q3}) reduces to
\begin{align}
(s-1)(|x-y|^2-|x|^2|y|^2(|x|^2-2x\cdot y+|y|^2))=(s-1)|x-y|^2(1-|x|^2|y|^2).\label{harmonic-q4}
\end{align}
We conclude that 
\begin{equation}\label{aim}
P=\frac{V_s(\rho)}{(1-|x|^2)} \frac{1-|x|^2|y|^2}{|x-y|^{2+N-2s}}=\frac{(1-|x|^2)^{s-2}(1-|y|^2)^{s-1}(1-|x|^2|y|^2)}{\Big|x|y|-\frac{y}{|y|}\Big|^{N}}=P_{s-1}(x,y),
\end{equation}
 as claimed.
 \medskip
 
We now consider the case $2s=N$. Since $s>1$ then $N\geq 3$. Note that $k_{N,s-1}=4(s-1)^2k_{N,s}$ and 
\begin{align*}
\cG_{s-1}(x,y)&=k_{N,s-1}|x-y|^{-2}\int_0^{\rho}\frac{v^{\frac{N}{2}-2}}{(v+1)^{\frac{N}{2}}}\ dv=4(s-1)k_{N,s} \frac{\rho^{s-1}}{(\rho+1)^{s-1}|x-y|^2}.
\end{align*}
On the other hand,
\begin{align*}
(-\Delta)\cG_{s}(x,y)&=-k_{N,s}\Delta \left(\int_0^{\rho}\frac{v^{\frac{N-2}{2}}}{(v+1)^{\frac{N}{2}}}\ dv\right)\\
&=4(s-1)k_{N,s}\frac{\rho^{s-1}(1-|y|^2)}{(\rho+1)^{s}|x-y|^4}\left[|y|^2-2x\cdot y+1-\frac{|x-y|^2}{1-|x|^2}\right].
\end{align*}
Hence,
\begin{align*}
&(-\Delta)\cG_{\frac{N}{2}}(x,y)=\cG_{\frac{N-2}{2}}(x,y) \\
&\quad +4(s-1)k_{N,s}\frac{\rho^{s-1}}{(\rho+1)^{s}|x-y|^4}\Bigg[(1-|y|^2)\Big[|y|^2-2x\cdot y+1-\frac{(1-|y|^2)}{\rho}\Big] -(\rho+1)|x-y|^2\Bigg],
\end{align*}
where,
\begin{align*}
(1-|y|^2)&\Big[|y|^2-2x\cdot y+1-\frac{(1-|y|^2)}{\rho}\Big] -(\rho+1)|x-y|^2\\
&=-|y|(|y|^2-2x\cdot y +|x|^2)-\frac{1-|y|^2}{1-|x|^2}|x-y|^2 = -|x-y|^2\Big(|y|^2+\frac{1-|y|^2}{1-|x|^2}\Big).
\end{align*}
Since $\rho+1=[x,y]^2|x-y|^{-2}$ we obtain that $-\Delta\ \cG_{s}=\cG_{s-1}-k_{N,s}4(s-1)P_{s-1}$ with $P_{s-1}$ as given by \eqref{Psm1} and the proof is finished.
\end{proof}

\section{Interchange of derivatives}

In the following we give assumptions on $u$ to guarantee that $(-\Delta)^{\sigma} (-\Delta)u=(-\Delta)(-\Delta)^\sigma u$ for $\sigma\in(0,1)$ in the pointwise sense, see \eqref{fraclaplace}. Let $H_u$ denote the Hessian of $u$.
\begin{lemma}\label{diffrep}
Let $V\subset \R^N$ open, $u:V\to\mathbb R^N$ such that $\|u\|_{C^2(V)}<\infty$, and $w:V\times \R^N\to\R$, $w(x,y):=2u(x)-u(x+y)-u(x-y)$. Then
\begin{align*}
w(x,y)=-\left[\int_0^1\int_0^1 H_u(x+(\tau-t)y)\ d\tau dt\right] y\cdot y\quad \text{ for all }x\in V,\ y\in \R^N,\ x\pm y\in V.
\end{align*}
In particular, $|w(x,y)|\leq \|u\|_{C^2(V)}|y|^2$ for all $x\in V$ and $y\in \R^N$ such that $x\pm y\in V.$
\end{lemma}
 \begin{proof}
Since $w(x,y)=u(x)-u(x+y)-(u(x)-u(x-y))$ we have by the Mean Value Theorem that $w(x,y)= \int_0^1[\nabla u(x+y-ty) - \nabla u(x-ty)]\ dt\cdot (-y)$. A second application of the Mean Value Theorem yields the result.
\end{proof}

The next proposition provides conditions to allow the interchange between derivatives and fractional Laplacians. The main difficulty in the proof relies on the fact that $u$ is allowed to have unbounded or discontinuous derivatives outside a domain $\Omega$. 

\begin{prop}\label{interchange-der}
Let $\Omega\subset \R^N$ open, $\sigma\in (0,1)$, and $u\in C^{3}(\Omega)\cap\cL^1_{\sigma}\cap W^{1,1}_{loc}(\mathbb R^N)$. If $\partial_1 u\in \cL^1_{\sigma}$, then $\partial_1(-\Delta)^\sigma u(x) = (-\Delta)^\sigma \partial_1u(x)$ pointwise for all $x\in \Omega$, where $(-\Delta)^{\sigma}u$ is evaluated as in (\ref{fraclaplace}). In particular, if $m\in \N_0$, $u\in C^{2m+2}(\Omega)\cap \cL^1_{\sigma}\cap W^{2m,1}_{loc}(\mathbb R^N)$, and $\partial^\alpha u\in \cL^1_{\sigma}$ for all $|\alpha|\leq 2m$, then 
\begin{align*}
(-\Delta)^{m+\sigma}u(x)=(-\Delta)^{\sigma}\,[\,(-\Delta)^m u(x)\,]=(-\Delta)^m\,[\,(-\Delta)^{\sigma}u(x)\,]\qquad \text{ for all }x\in \Omega. 
\end{align*}
\end{prop}

\begin{proof}
Let $u\in C^{3}(\Omega)\cap \cL^1_{\sigma}\cap W^{1,1}_{loc}(\mathbb R^N)$ and $\partial_1 u\in C^{2}(\Omega)\cap \cL^1_{\sigma}$. In the following all derivatives $\partial_1$ are taken with respect to $x$. By \cite[Lemma 2.1]{FW15} we have that
\begin{align*}
 (-\Delta)^\sigma u(x)=c_{N,\sigma}P.V. \int_{\R^N}\frac{u(x)-u(y)}{|x-y|^{N+2\sigma}}\ dxdy = c_{N,\sigma}\int_{\R^N}\frac{2u(x)-u(x-y)-u(x+y)}{|y|^{N+2\sigma}}\ dxdy,
\end{align*}
where the integral on the right does not have a principal value (cf. \cite[Lemma 3.2]{NPV11}). Let $H:\Omega\times\R^N\setminus\{0\}\to \R$ and $h_t:\Omega\times \R^N\setminus\{0\}\to\R$ be given by 
\begin{align*}
 H(x,y):=\frac{2 u(x)-u(x+y)- u(x-y)}{|y|^{N+2\sigma}},\quad h_t(x,y):= \frac{H(x+te_1,y)-H(x,y)}{t},\quad t\in\R\backslash\{0\}.
\end{align*}

Fix $x\in\Omega$ and $V$ an open set with $\overline{V}\subset \Omega$ and $x\in V$. Let $T,\varepsilon\in(0,1)$ such that $x+y+te_1\in V$ for all $0<|t|<T$ and $|y|<\varepsilon.$ Set $U:=B_{\varepsilon}(0)$. We show separately that 
\begin{align}
&\lim_{t\to 0}\int_{U} h_t(x,y)\ dy= \int_{U}\partial_1 H(x,y)\ dy\quad \text{ and }\label{conc1}\\
&\lim_{t\to 0}\int_{\R^N\backslash U} h_t(x,y)\ dy= \int_{\R^N\backslash U}\partial_1 H(x,y)\ dy\label{conc2}.
\end{align}

By the Mean Value Theorem, for every $ 0<|t |<T$ there is $|t_0|<t$ and $\xi:=x+t_0e_1\in V$ such that 
$h_t(x,y)=\partial_1H(\xi,y)$ for $y\in U$.  Then, by Lemma \ref{diffrep}, $|\partial_1H(\xi,y)|\leq  \|u\|_{C^{3}(V)}|y|^{-2\sigma-N+2}\in L^1(U)$. Thus, by the Dominated Convergence Theorem, $\partial_1 H(x,\cdot)\in L^1(U)$ and \eqref{conc1} holds.

Moreover, if $A:=\{|y-ste_1-x|\geq \varepsilon\}$, then
\begin{align*}
 \left|\frac{\partial_1 u(y)}{|y-ste_1-x|^{N+2\sigma}}1_A(y)\right|\leq\frac{|\partial_1 u(y)|}{1+|y|^{N+2\sigma}}\frac{1+|y|^{N+2\sigma}}{|y-ste_1-x|^{N+2\sigma}}1_A(y)\leq K\frac{|\partial_1 u(y)|}{1+|y|^{N+2\sigma}}=:f(y),
\end{align*}
where $K>0$ is a constant depending only on $V,N,\varepsilon$, and $\sigma$. Since $f\in L^1(\R^N)$ then, by the Dominated Convergence Theorem,
\begin{align*}
\lim_{t\to 0}\int_{\R^N}\int_0^1 \frac{\partial_1 u(y)}{|y-ste_1-x|^{N+2\sigma}}1_{\{|y-ste_1-x|\geq \varepsilon\}}(y)\ dsdy=\int_{\R^N} \frac{\partial_1 u(y)}{|y-x|^{N+2\sigma}}1_{\{|y-x|\geq \varepsilon\}}(y)\ dy
\end{align*}
or equivalently,
\begin{align}\label{U1}
\int_{\R^N} \frac{\partial_1 u(x\pm y)}{|y|^{N+2\sigma}}1_{\{|y|\geq \varepsilon\}}\ dy&= \lim_{t\to 0}\int_{\R^N}\int_0^1 \frac{\partial_1 u(ste_1+x\pm y)}{|y|^{N+2\sigma}}1_{\{|y|\geq \varepsilon\}}\ dsdy\nonumber\\
&= \lim_{t\to 0}\int_{\R^N}\frac{u(x+te_1\pm y)-u(x\pm y)}{|y|^{N+2\sigma}}1_{\{|y|\geq \varepsilon\}}\ dy.
\end{align}

Since it trivially holds that
\begin{align}\label{U2}
 \lim_{t\to0}\frac{1}{t}\int_{\R^N\backslash U}\frac{u(x+te_1)-u(x)}{|y|^{N+2\sigma}}\ dy
 =\int_{\R^N\backslash U}\frac{\partial_1u(x)}{|y|^{N+2\sigma}}\ dy,
\end{align}
then \eqref{conc2} follows from \eqref{U2} and \eqref{U1}.
\end{proof}

To perform the integration by parts we use the following standard regularity result. 
\begin{lemma}\label{reg-c2m} Let $\Omega\subset \R^N$ open, $m\in \N$, $\sigma\in(0,1)$, $s=m+\sigma$, and let $u\in C^{2s+\alpha}_{loc}(\Omega)\cap C^s(\R^N)\cap\cL^1_s$ for some $\alpha>0$. Then $(-\Delta)^{\sigma}u\in C^{2m}_{loc}(\Omega)\cap {C^{m-\sigma}(\mathbb R^N)}$.
\end{lemma}
The proof can be done by arguing as in the proof of \cite[Propositions 2.6 and 2.7]{S07} and hence we omit it. 

\begin{lemma}\label{ibyp}
Let $\sigma\in(0,1)$, $m\in\N,$ and $s=m+\sigma>1$. If $u\in W^{2,1}(B)$ satisfies $u=\nabla u=0$ on $\partial B$ in the trace sense, then
 \begin{align}\label{ibyp:eq}
  \int_{B} u\, (-\Delta)^s \varphi\ dx = \int_{B} -\Delta u\,(-\Delta)^{s-1} \varphi\ dx.
 \end{align}
This is in particular the case if $u\in W^{2,1}(\R^N)$ with $\supp\ u\subset \overline{B}$.  If $u\in C^{2s+\alpha}_{loc}(B)\cap C_0^{s}(B)$ for some $\alpha\in(0,1)$, then
 \begin{align}\label{ibyp:eq:2}
  \int_{\R^N} u\, (-\Delta)^s \varphi\ dx= \int_{\R^N} (-\Delta)^{m}(-\Delta)^\sigma u\, \varphi\ dx\qquad \text{for all $\varphi\in C^\infty_c(B)$,}
 \end{align}
and if $u\in \cH^s_0(B)$ then $\int_{\R^N} u\, (-\Delta)^s \varphi\ dx= \cE_s(u\,,\,\varphi)$ for all $\varphi\in \cH^{s}_0(B)$.
\end{lemma}
\begin{proof}
Equality \eqref{ibyp:eq} follows from two integrations by parts, since $u\equiv 0$ in $\R^N\backslash B$ and $\nabla u =0$ on $\partial B$.  For \eqref{ibyp:eq:2}, note that $u\in C^{2s+\alpha}_{loc}(B)\cap C_0^{s}(B)$ implies that $(-\Delta)^\sigma u \in C^{2m}_{loc}(B)\cap {C^{m-\sigma}(\mathbb R^N)}$ by Lemma \ref{reg-c2m}, since $s>1$.  Moreover, since $u\equiv 0$ in $\R^N\backslash B$, there is $C>0$ such that $|\Delta^\sigma u(x)|\leq C({1+|x|^{N+2\sigma}})^{-1}$ for all $x\in \R^N.$  In particular, $(-\Delta)^\sigma u \in L^2(\R^N)$. Using Fourier transform, integration by parts, and the fact that $\varphi$ has compact support on $B$, we obtain
\begin{align*}
  \int_{\R^N} u(x)(-\Delta)^s \varphi(x)\ dx = \int_{\R^N} (-\Delta)^\sigma u(x) (-\Delta)^{m}\varphi(x)\ dx=\int_{\R^N} (-\Delta)^{m}(-\Delta)^\sigma u(x) \varphi(x)\ dx.
\end{align*}
The last claim follows from Lemma \ref{existence}.
\end{proof}

\begin{lemma}\label{ibyp:RN}
 Let $s>1$ and $u\in H^{2}_{loc}(\R^N)$ such that $\Delta u \in \cL^1_{s-1}$. Then,
 \begin{align}\label{ibyp:RN:eq}
  \int_{\R^N} u\,(-\Delta)^s \varphi\ dx = \int_{\R^N} -\Delta u\,(-\Delta)^{s-1} \varphi\ dx\qquad \text{ for all }\varphi\in C^\infty_c(\R^N).
 \end{align}
 \end{lemma}
\begin{proof}
 Fix $\psi:=(-\Delta)^{s-1}\phi$. Then $\psi\in C^\infty(\R^N)$ (see by \cite[Proposition 2.7]{S07}) and, by Lemma \ref{boundedevaluationexistence} and Proposition \ref{interchange-der}, there is $K=K(\varphi,N,s)>0$ such that 
 \begin{align}\label{dpsi}
 |\psi(x)|+ |\nabla\psi(x)|\leq \frac{K}{1+|x|^{N+2(s-1)}}\quad \text{ for all }x\in\R^N.
 \end{align}
 Let $(\eta_n)_{n\in\N}\subset C^\infty(\R^N)$ satisfy
 \begin{align}\label{etap}
  0\leq \eta_n\leq 1,\quad \eta_n\equiv 1\quad \text{ in }B_n(0),\quad  
 \eta_n\equiv 0\quad \text{ in }\R^N\backslash B_{n+1}(0),\quad \|\eta_n\|_{C^2(\R^N)}<C
  \end{align}
  for some $C>0$ independent of $n$, and set $\psi_n:=\eta_n \psi\in C^\infty_c(\R^N)$. Then $\psi_n\to \psi$ in $L^2(\R^N)$ and 
  $-\Delta\psi_n=-\Delta \psi\eta_n-\nabla\eta_n\nabla\psi_n-\psi\Delta\eta_n\to -\Delta\psi=(-\Delta)^s\varphi$ in $L^2(\R^N)$,
  by \eqref{etap}, \eqref{dpsi}, and Proposition \ref{interchange-der}.  Therefore, 
 \begin{align*}
  \int_{\R^N} u\,(-\Delta)^s \varphi\ dx=\lim_{n\to\infty}\int_{\R^N} u\,(-\Delta)\psi_n\ dx = \lim_{n\to\infty}\int_{\R^N} -\Delta u\,\psi_n\ dx=\int_{\R^N} -\Delta u\,(-\Delta)^{s-1} \varphi\ dx,
 \end{align*}
 as claimed.
\end{proof}



\end{document}